\documentclass[11pt]{amsart}
\usepackage{amssymb}
\usepackage{amsmath}
\usepackage{mathrsfs}
\usepackage{amsthm}
\usepackage{amsfonts}
\usepackage{txfonts}
\usepackage{enumerate}
\usepackage{hyperref}
\usepackage{latexsym,amssymb}
\usepackage{color}
\usepackage{stmaryrd}
\usepackage{makecell} 

\usepackage{indentfirst}
\usepackage{graphicx}
\usepackage{float}
\usepackage{pgf,tikz}
\usepackage{mathrsfs}
\usetikzlibrary{arrows}

\makeatletter
\@namedef{subjclassname@1991}{Subject}
\@namedef{subjclassname@2000}{Subject}
\@namedef{subjclassname@2010}{Subject}
\@namedef{subjclassname@2020}{{\rm 2020} Mathematics Subject Classification}
\makeatother

\definecolor{ttttff}{rgb}{0.2,0.2,1.}
\definecolor{ttffcc}{rgb}{0.2,1.,0.8}
\definecolor{qqqqff}{rgb}{0.,0.,1.}

\definecolor{zzttqq}{rgb}{0.6,0.2,0.}
\definecolor{qqqqff}{rgb}{0.,0.,1.}

\allowdisplaybreaks

\hypersetup{
        colorlinks   = true,
        citecolor    = blue,
        linkcolor    = blue,
        urlcolor     = blue
}

\def\rr{{\mathbb R}}
\def\rn{{\mathbb{R}^n}}

\def\zz{{\mathbb Z}}
\def\cc{{\mathbb C}}

\def\nn{{\mathbb N}}

\def\cj{{\mathcal J}}

\def\cl{{\mathcal L}}

\def\fz{\infty }
\def\az{\alpha}

\def\ez{\epsilon}

\def\sz{\sigma}

\def\ls{\lesssim}
\def\gs{\gtrsim}

\def\gfz{\genfrac{}{}{0pt}{}}

\def\loc{{\mathop\mathrm{\,loc\,}}}
\def\supp{\mathop\mathrm{\,supp\,}}

\def\az{\alpha}

\def\dint{\displaystyle\int}

\newtheorem{theorem}{Theorem}[section]
\newtheorem{lemma}[theorem]{Lemma}
\newtheorem{corollary}[theorem]{Corollary}
\newtheorem{proposition}[theorem]{Proposition}

\theoremstyle{definition}
\newtheorem{remark}[theorem]{Remark}
\newtheorem{definition}[theorem]{Definition}
\renewcommand{\appendix}{\par
   \setcounter{section}{0}%
   \setcounter{subsection}{0}%
   \setcounter{subsubsection}{0}%
   \gdef\thesection{\@Alph\c@section}%
   \gdef\thesubsection{\@Alph\c@section.\@arabic\c@subsection}%
   \gdef\theHsection{\@Alph\c@section.}%
   \gdef\theHsubsection{\@Alph\c@section.\@arabic\c@subsection}%
   \csname appendixmore\endcsname
 }

\numberwithin{equation}{section}

\definecolor{qqqqff}{rgb}{0,0,1}
\definecolor{ffqqqq}{rgb}{1,0,0}

\def\einf{\mathop{\mathrm{essinf}}}

\newcounter{rea}
\allowdisplaybreaks

\begin{document}
\title[Nontrivial nonnegative weak solutions to fractional $p$-Laplace
inequalities]
{Nontrivial nonnegative weak solutions \\
to fractional $p$-Laplace
inequalities}

\author{Liguang Liu}
\address{School of Mathematics,
	Renmin University of China,
	Beijing 100872, People's Republic of China}
\email{liuliguang@ruc.edu.cn}

\thanks{L. Liu is supported by the National Natural Science Foundation of China (\# 12371102).}

	\date{}
		
	\subjclass[2020]{Primary: 35B09;
Secondary: 35B53, 35A08, 35A23, 47G20.}
	
	\keywords{Liouville theorem, fractional $p$-Laplace operator, fractional Sobolev space, Riesz potential.}


\begin{abstract}
For the nonlocal quasilinear fractional
$p$-Laplace operator $(-\Delta)^s_p$ with $s\in (0,1)$ and $p\in(1,\infty)$,
we investigate the
nonexistence and existence of nontrivial nonnegative solutions $u$ in the local fractional Sobolev space $W_{\rm loc}^{s,p}(\mathbb R^n)$
that satisfies the inequality
$(-\Delta)^s_p u\ge u^q$
weakly in $\mathbb R^n$, where $q\in(0,\infty)$.
The approach taken in this paper is mainly based on some delicate analysis of the fundamental solutions to the fractional $p$-Laplace operator $(-\Delta)^s_p$.
\end{abstract}


\maketitle

\arraycolsep=1pt
\allowdisplaybreaks

\tableofcontents

\section{Introduction}

It is known that the classical Liouville
theorem says that bounded entire functions on the complex plane $\cc$ must be constant functions.
In the context of real analysis, Liouville theorem
 states that any real harmonic function $u$ on $\rn$ must be a constant if it is bounded either from above
or below, where harmonic means that $\Delta u=0$, with $\Delta=\sum_{j=1}^n\partial_{x_j}^2$ being the Laplace operator on $\rn$.
Suppose that $q\in(0,\infty)$.
Under $n>2$, Gidas and Spruck \cite{Gidas1980,GidasSpruck1981CPAM}
prove that
\begin{align}\label{eq-GS}
-\Delta u = u^q
\end{align}
has no nontrivial nonnegative solutions in $\rn$ provided that $q<\frac{n+2}{n-2}$.
For the differential inequality
\begin{align}\label{eq-NS}
-\Delta u \ge u^q,
\end{align}
Ni and Serrin \cite{NiSerrin1986CPAM} prove that under $n>2$ it has no nontrivial nonnegative solutions in $\rn$  if and only if $q\le \frac{n}{n-2}$.
Generalizations of these results to general differential inequalities or equations involved the
 $p$-Laplace operator $\Delta_p u:=-\mathrm{div}(|\nabla u|^{p-2}\nabla u)$ are due to Mitidieri and Pokhozhaev  \cite{MitidieriPokhozhaev1999, MitidieriPokhozhaev1998, MitidieriPokhozhaev2001}. Indeed,  the Lane-Emden $p$-Laplace equation
\begin{align}\label{eq-MP}
\Delta_p u = u^q
\end{align}
has no nontrivial nonnegative solutions in $\rn$ provided that either $p\ge n$ or $1<p<n$ and $q<\frac{n(p-1)+p}{n-p}$. Moreover, the differential inequality
\begin{align}\label{eq-SZ}
\Delta_p u 
\ge u^q
\end{align}
has no nontrivial nonnegative solutions in $\rn$ provided that either $p\ge n$ or $1<p<n$ and $q\le\frac{n(p-1)}{n-p}$. The above results can also be  found in the work of Serrin and Zou \cite{SerrinZou2002Acta}.
A key technique that is often used in these literature is to take a special test function of the form
$u^{-a}\psi^b$, with $a,b>0$ and $\psi$ being a smooth cutoff function.
In addition, there is another method of obtaining Liouville theorems involved the $p$-Laplace operator, which is mainly based on the comparison method and Harnack inequality; see, for example, \cite{CaristiDAmbrosioMitidieri2008,CaristiMitidieriPokhozhaev2009} or \cite{ArmstrongSirakov2011CPDE}.
For  related works on various Liouville type theorems for differential equations or inequalities of $p$-Laplace operator, on exterior domains, on the entire Euclidean space $\rn$ or on manifolds,
we refer the reader to  \cite{BidautPohozaev2001JAM, Grigoryan1985MS, GrigoryanSun2014CPAM, Sun2015PAMS, Sun2016PJM,
FilippucciSunZheng2024JAM, Filippucci2009NA} and the references therein.

Let us turn to the  nonlocal case. For any $s\in(0,1)$,  Wang and Xiao \cite{WangXiao2016CCM} (see also
\cite{WangXiao2016AIHP}) prove that the fractional order differential inequality
\begin{align*}
(-\Delta)^{s} u \ge u^q
\end{align*}
has no nontrivial nonnegative solutions in $\rn$ if and only if $q\le \frac{n}{n-2s}$, which is a
fractional  generalization of \cite{GidasSpruck1981CPAM, NiSerrin1986CPAM}.  The main approach taken in
\cite{WangXiao2016CCM} is to utilize
Caffarelli-Silvestre's extension technique (see \cite{CaffarelliSilvestre2007CPDE}) that reduces the {\it nonlocal} operator $(-\Delta)^{s}$ to a {\it local} differential operator on the upper half space.
For the pointwise fractional differential equation
\begin{align*}
(-\Delta)^{s} u = u^q
\end{align*}
on the upper half space $\mathbb R^n_+$ subject to the boundary condition $u=0$ on $\rn\setminus\mathbb R^n_+$,
Chen, Fang and Yang \cite{ChenFangYang2015AdvMath} prove that if $u$ is a nonnegative locally bounded solution then $u\equiv 0$ on $\rn$ provided that $q\le \frac{n+2s}{n-2s}$.

In accordance to the above discussions, the main aim of this paper is to study Liouville properties related to the nonlocal quasilinear {\it fractional $p$-Laplace operator}
\begin{equation}\label{eq-Delta-sp} (-\Delta)^s_pu(x)=\mathrm{p.v.}\int_{\rn}\frac{|u(x)-u(y)|^{p-2}(u(x)-u(y))}{|x-y|^{n+sp}}\,dy,
\end{equation}
where $s\in(0,1)$ and $p\in(1,\infty)$.
In this paper, we are interested in finding the ranges of $s,p$ and $q$ for which the fractional $p$-Laplace inequality
\begin{align}\label{eq-sp>=}
(-\Delta)^s_pu\ge u^q
\end{align}
has no nontrivial nonnegative weak solutions.
The suitable class of function spaces for the solutions $u$ are  the following fractional Sobolev spaces (see, for example, \cite{DiNezzaPalatucciValdinoci2012}).

\begin{definition} Let $\Omega\subseteq\rn$ be an open set. For any $s\in (0,1)$ and $p\in[1,\infty)$,  define the {\it fractional Sobolev space $W^{s,p}(\Omega)$} to be the space of all $f\in L^p(\Omega)$ such that
\begin{align*}
[f]_{W^{s,p}(\Omega)}=\left(\int_\Omega\int_\Omega \frac{|f(x)-f(y)|^p}{|x-y|^{n+sp}}\,dy\,dx\right)^\frac 1p<\infty.
\end{align*}
For any $f\in W^{s,p}(\Omega)$, define the norm
\begin{align*}
\|f\|_{W^{s,p}(\Omega)}:= \|f\|_{L^p(\Omega)}+[f]_{W^{s,p}(\Omega)}.
\end{align*}
Denote by $W_\loc^{s,p}(\Omega)$ the space of all functions $u$ such that $u\in W^{s,p}(D)$ for all open sets $D$ satisfying $D\Subset U$ (the closure $\overline D$ is a compact subset of $U$).
If $\Omega=\rn$, then we write $W^{s,p}(\rn)$
or $W_\loc^{s,p}(\rn)$.
\end{definition}

The fractional $p$-Laplace operator $(-\Delta)^s_p$  in \eqref{eq-Delta-sp}  has attracted lots of attention in analysis, partial differential equations and geometry; see, for example, \cite{DiNezzaPalatucciValdinoci2012, ShiXiao2017CVPDE,  KorvenpaaKuusiPalatucci2017MathAnn, KorvenpaaKuusiLindgren2019JMPA, DiCastroKuusiPalatucci2014JFA, KuusiMingioneSire2015APDE,KuusiMingioneSire2015CMP,
LindgrenLindqvist2014CVPDE,DiCastroKuusiPalatucci2016AIHP} and etc.
Indeed, this nonlocal nonlinear operator $(-\Delta)^s_p$  arises naturally when  computing the critical points of the fractional $W^{s,p}(\rn)$-energy
$$
\frac{d}{dt}\left( [u+t\phi]_{W^{s,p}(\rn)}^p\right)\bigg|_{t=0},
$$
where $\phi\in C_c^\infty(\rn)$, i.e., the space of infinitely differentiable functions on $\rn$ with compact support.
Obviously, $(-\Delta)^s_p$ extends the fractional Laplace operator $(-\Delta)^s$ since at the case  $p=2$ one has
$$(-\Delta)^s_2=(-\Delta)^s.$$
Moreover,  $(-\Delta)^s_p$
is  a natural generalization of the classical $p$-Laplace operator $\Delta_p$,
because $(-\Delta)^s_p$
connects the identity operator and  $\Delta_p$
in the following way:  for  functions $u\in C_c^\infty(\rn)$,
$$
\lim_{s\to 0_+} s \left\langle(-\Delta)^s_pu,\, u\right\rangle_{L^2(\rn)} = c_{p, n}\left\langle u,\, u\right\rangle_{L^2(\rn)}
$$
and
$$
\lim_{s\to 1_-}(1-s)\left\langle(-\Delta)^s_pu,\, u\right\rangle_{L^2(\rn)} = C_{p,n}\left\langle\Delta_p u,\, u\right\rangle_{L^2(\rn)},
$$
where $c_{p,n}$ and $C_{p,n}$ are
positive constants depending only on $p$ and $n$.
The latter two convergences are consequences of the following  well-known limiting behaviours
of fractional Sobolev functions (see \cite{MazyaShaposhnikova2002JFA} and \cite{BourgainBrezisMironescu2002JAM, BourgainBrezisMironescu2001}):
$$
\lim_{s\to 0_+} s[u]_{W^{s,p}(\rn)}^p =2c_{p,n}\|u\|_{L^p(\rn)}^p$$
and
$$
\lim_{s\to 1_-} (1-s)[u]_{W^{s,p}(\rn)}^p =2C_{p,n}\|\nabla u\|_{L^p(\rn)}^p.
$$
 Based on these discussions, we emphasize that \eqref{eq-sp>=} extends both \eqref{eq-NS} and \eqref{eq-SZ}.

Now, we state the  main result of this paper, which concerns nonexistence and existence  of weak local solutions to $(-\Delta)^s_pu\ge u^q$ in \eqref{eq-sp>=}.

\begin{theorem}
		\label{thm-Liouville}
Let $n\in\nn$, $s\in (0,1)$, $p\in (1,\infty)$ and $q\in(0,\infty)$. Assume that $0\le u\in W_\loc^{s,p}(\rn)$  satisfies the fractional $p$-Laplace inequality
		\begin{equation}
		\label{eq-p-Lapineq}
		(-\Delta)^s_pu\ge u^q\ \ \text{weakly in}\ \ \mathbb R^n,
		\end{equation}
that is, for all $0\le\phi\in C_c^\infty(\rn)$,
	\begin{equation}
		\label{eq-phi>} \frac 12\int_{\rn}\int_{\rn}\frac{|u(x)-u(y)|^{p-2}(u(x)-u(y))(\phi(x)-\phi(y))}{|x-y|^{n+sp}}\,dy\,dx\ge\int_{\mathbb R^n}u(x)^q\phi(x)\,dx.
\end{equation}
Then the following statements hold:
\begin{itemize}
\item[\rm (i)] If $sp\ge n$, then \eqref{eq-p-Lapineq} has only the trivial solution $u=0$ a.e. on $\mathbb R^n$.
					
\item[\rm (ii)] If $sp<n$ and $q\le\frac{n(p-1)}{n-sp}$, then \eqref{eq-p-Lapineq} has only the trivial solution $u=0$ a.e. on $\rn$.

\item[\rm (iii)] If $sp<n$ and $q>\frac{n(p-1)}{n-sp}$, then \eqref{eq-p-Lapineq} has a nontrivial positive pointwise solution on $\rn$.
\end{itemize}
	
	\end{theorem}

\begin{remark} Let us give  three comments on Theorem \ref{thm-Liouville}.

\begin{enumerate}
\item[\rm (i)]
Theorem \ref{thm-Liouville} not only generalizes, but also recovers the classical Liouville theorems for both the classical Laplace operator $\Delta$  as studied in \cite{Gidas1980, GidasSpruck1981CPAM, NiSerrin1986CPAM} and
and  the classical $p$-Laplace operator $\Delta_p$ as explored in  \cite{MitidieriPokhozhaev1999, MitidieriPokhozhaev1998, SerrinZou2002Acta}. It is important to note that Theorem \ref{thm-Liouville},
with  $p=2$ therein,
also recovers the Liouville theorem for the fractional Laplace operator $(-\Delta)^s$,
as established by Wang and Xiao \cite{WangXiao2016AIHP, WangXiao2016CCM}.

\item[\rm (ii)]
As a consequence of Theorem \ref{thm-Liouville},
if $0\le u\in W_\loc^{s,p}(\rn)$  satisfies the fractional $p$-Laplace
equation
\begin{align}\label{eq-sp=}
(-\Delta)^s_pu= u^q \ \ \text{weakly in}\ \ \mathbb R^n,
\end{align}
 that is, for all $0\le\phi\in C_c^\infty(\rn)$,
		\begin{equation*}
	 \frac12\int_{\rn}\int_{\rn}\frac{|u(x)-u(y)|^{p-2}(u(x)-u(y))
(\phi(x)-\phi(y))}{|x-y|^{n+sp}}\,dy\,dx=\int_{\mathbb R^n}u(x)^q\phi(x)\,dx,
	\end{equation*}
then $u=0$ a.e. on $\mathbb R^n$ provided that  either $sp\ge n$ or $sp<n$ and $q\le \frac{n(p-1)}{n-sp}$.
Based on the relations in \eqref{eq-GS} and \eqref{eq-MP}, it is expected that when
$$
\frac{n(p-1)}{n-sp} <q< \frac{np}{n-sp}-1,
$$
any nonnegative function $u\in W_\loc^{s,p}(\rn)$  that weakly solves the equation \eqref{eq-sp=} must be identically zero almost everywhere on $\rn$.
It is both interesting and much more involved  (see, e.g., \cite{Zou1995DIE, SerrinZou2002Acta}) to investigate the behavior of weak solutions to  \eqref{eq-sp=} when the exponent exceeds the critical value, as this  may give rise to significant  phenomena.
However, these topics are beyond the scope of the current paper and will be explored in future work.


\item[\rm (iii)]
Another way to extend the
$p$-Laplace operator $\Delta_p$
  to the fractional setting is to define, for
$s\in(0,1)$, the following version of the fractional $p$-Laplace operator
\begin{align*}
 \cl_{s,p}u:=- \mathrm{div}^s (|\nabla^s u|^{p-2}\nabla^s u),
\end{align*}
where $\nabla^s:= \nabla (-\Delta)^\frac {s-1}2$ is the {\it fractional gradient operator}, and $\mathrm{div}^s:=\mathrm{div}(-\Delta)^{\frac{s-1}{2}}$ is  the
{\it fractional divergence operator}.
If $p=2$, then
$$\cl_{s,2}u=- \mathrm{div}^s (\nabla^s u)=(-\Delta)^{s}u.$$
If $s=1$, then  $\nabla^s=\nabla $ and $\mathrm{div}^s =\mathrm{div}$, which gives
$$\cl_{1,p}u=-\mathrm{div}(|\nabla u|^{p-2}\nabla u)=\Delta_p.$$
Thus $\cl_{s,p}$  serves as a generalization of both the local quasilinear
$p$-Laplace operator $\Delta_p$
and the nonlocal linear fractional operator $(-\Delta)^{s}$.
Related studies on $\cl_{s,p}$
  in the context of PDEs can be found in \cite{SchikorraShiehSpector2018CCM, ShiehSpector2015ACV, ShiehSpector2018ACV}. It is proved in \cite[Theorem 1.5]{LiuSunXiao2024MathAnn} that, under the condition
  $p\ge \frac{2n-s}{n}$ and some integrability assumptions on $\nabla^su$,  the fractional  differential inequality
\begin{align*}
\cl_{s,p} u \ge u^q
\end{align*}
has no nontrivial nonnegative weak solutions in $W^{s,p}(\rn)$ if and only if  either $sp\ge n$ or $sp<n$ and $q\le\frac{n(p-1)}{n-sp}$. The integrability assumption on $\nabla^su$ can not be
dropped  because comparison principle is not known for $\cl_{s,p}$.
In comparison, Theorem \ref{thm-Liouville}(ii) does not require any additional integrability assumptions on $u$.
\end{enumerate}
\end{remark}

The key ingredient in the proof of Theorem \ref{thm-Liouville} is the subsequently presented fundamental solution to  the fractional $p$-Laplace operator $(-\Delta)^s_p$.

\begin{theorem}\label{thm-solution}
		Let $n\in\nn$, $s\in(0,1)$, $p\in (1, \infty)$ such that $sp\le n$. For any $x\in\rn\setminus\{0\}$, define
		\begin{equation}
		\label{eq-u0}
		u_{0}(x)=\begin{cases}
		|x|^\frac{sp-n}{p-1}&\ \ \text{as}\ \ sp<n;\\
		\ln\frac 1{|x|}&\ \ \text{as}\ \ sp=n.
		\end{cases}
		\end{equation}
Denote by $\delta_0$  the Dirac measure at the origin of $\rn$.		Then there exists a constant  $c_{n,s,p}^\ast$  such that
\begin{equation}
		\label{eq-U}
u=\begin{cases}
(c_{n,s,p}^\ast)^{-\frac1{p-1}} u_{0}\quad&\text{as}\ \ c_{n,s,p}^\ast\ge0;\\
-(-c_{n,s,p}^\ast)^{-\frac1{p-1}} u_{0}\quad&\text{as}\ \ c_{n,s,p}^\ast<0,
\end{cases}
\end{equation}
is a solution to
		\begin{equation*}
		(-\Delta)^s_pu=\delta_0\quad \text{weakly in}\ \rn
\end{equation*}
in the following sense:\, for all $\phi\in C_c^\infty(\rn)$,
	\begin{equation*}		\frac12\int_{\rn}\int_{\rn}\frac{|u(x)-u(y)|^{p-2}(u(x)-u(y))
(\phi(x)-\phi(y))}{|x-y|^{n+sp}}\,dy\,dx=\phi(0).
		\end{equation*}
Indeed, the constant 	$c_{n,s,p}^\ast$ is given by
\begin{equation}\label{eq-cnsp}
c_{n,s,p}^\ast=  \frac12 \int_{\rn}\int_{\rn}\frac{|u_0(x)-u_0(y)|^{p-2}(u_0(x)-u_0(y))\left(|x-{\bf e}_1|^{s-n}-|y-{\bf e}_1|^{s-n}\right)}{|x-y|^{n+sp}} \,dy\,dx.
\end{equation}
\end{theorem}

The paper is organized as follows. In Section \ref{ss-solution}, we prove Theorem \ref{thm-solution} through extensive technical computations and delicate estimates. Section \ref{sec3} focuses on the proofs of parts (i) and (ii) of Theorem \ref{thm-Liouville}. The main strategies used are Theorem \ref{thm-harnack}
(to treat the cases $sp\ge n$ or $sp<n$ and $q<\frac{n(p-1)}{n-sp}$) and  Theorem \ref{thm-mr} (to treat the case $sp<n$ and $q=\frac{n(p-1)}{n-sp}$),
both of which are established in Section \ref{ss3.3}. Then (i) and (ii) of Theorem \ref{thm-Liouville} are  proved in Section \ref{ss3.4}. In Section \ref{sec5}, we establish Theorem \ref{thm-Liouville}(iii) by applying perturbation and smoothness techniques to the fundamental solution from Theorem \ref{thm-solution}.


\medskip

\noindent{\bf Notation.\,}
Throughout the paper, we adopt the following notation and notions:

\begin{itemize}

\item Let  $\mathbb{N}=\{1,2,\cdots\}$, $\zz=\{0,\pm1,\pm2,\dots\}$ and $\mathbb{Z}_{+}=\mathbb{N}\cup\{0\}$.

\item  For any multi-index $\nu=(\nu_{1},\dots,\nu_{n})\in\mathbb{Z}^{n}_{+}$, denote by $|\nu|=\nu_1+\cdots+\nu_n$ and write $D^\nu=\partial^{\nu_{1}}_{1}\cdots\partial^{\nu_{n}}_{n}$, where $\partial_{j}=\partial_{x_{j}}=\frac{\partial}{\partial x_{j}}$ for any $j\in\{1,\dots,n\}$.

\item For any $j\in\{1,2,\dots,n\}$, denote by ${\bf e}_j$ the unit vector in $\rn$ whose $j$-th entry is $1$ but all other entries are zero.

  \item Denote by $\nu_n$ the Lebesgue measure of the unit ball in $\rn$. Let $\omega_{n-1}=n\nu_n$ be the  surface measure of the unit ball in $\rn$.

\item $C^\infty(\Omega)$ denotes the  space of all infinitely differentiable functions on $\Omega$;
$C_c^\infty(\Omega)$ denotes the  space of functions in $C^\infty(\Omega)$
with compact support contained in $\Omega$.

\item For any subset $E$ of $\rn$, denote by $\overline E$ and $\mathring E$ the closure and interior of $E$, respectively. Meanwhile, set $E^c:=\rn\setminus E$.

\item For any subset $E$ of $\rn$ with finite Lebesgue measure and for any locally integrable function $f$ on $\rn$, we use the notation
    $$\fint_E f(x)\,dx:=\frac1{|E|}\int_E  f(x)\,dx.$$

\item For any $r\in (0,\infty)$, let  $\mathbb B_r:=\{x\in\rn:\, |x|< r\}$ and $\overline{\mathbb B}_r:=\{x\in\rn:\, |x|\le r\}$.

\item For any two open sets $U$ and $V$ in $\rn$, the symbol $U\Subset V$ means that
$\overline U$ is compact and $\overline U\subseteq V$.

\item Set $u_+:=\max\{u,\,0\}$ and $u_-:=\max\{-u,\,0\}$.

  \item The letters $C$ and $c$ are used to denote positive constants that are independent of the
variables in question, but may vary at each occurrence.
The symbol $u\ls v$ (resp., $u \gs v$) means that $u \le Cv$
(resp., $u \ge Cv$) for a positive constant $C$ independent of the main parameters involved. We
write $u \simeq     v$ if $u \ls  v \ls  u$.

\end{itemize}

\section{Fundamental solution}\label{ss-solution}

This section is devoted to the proof of Theorem \ref{thm-solution}.
In Section \ref{ss2.1}, we demonstrate that the action of $(-\Delta)^s_p$ on
functions $u$ with sufficiently good regularity and integrability (for instance,
the fundamental solution  away from the origin) can be continuous. Then, in
Section \ref{ss2.2}, we establish a very delicate estimate for $u_0$ in
\eqref{eq-u0} (see Proposition \ref{prop1} below). 
Finally, in Section \ref{ss2.3}, 
we provide the proof of Theorem \ref{thm-solution}. 

\subsection{Behaviour of the fractional $p$-Laplace operator}\label{ss2.1}

Let us start with the definitions of tail spaces and H\"older spaces.

\begin{definition}
 Let $s\in (0,1)$ and $p\in(1,\infty)$. Define the   {\it tail space}
\begin{align}\label{eq-Lsp}
L_{sp}^{p-1}(\rn):=\left\{
u\in L_\loc^{p-1}(\rn):\ \||u\||_{L_{sp}^{p-1}(\rn)}:=\int_\rn \frac{|u(x)|^{p-1}}{(1+|x|)^{n+sp}}\,dx<\infty
\right\}.
\end{align}
In particular,  $L^\infty(\rn)\subseteq L_{sp}^{p-1}(\rn)$ and $L^{p}(\rn)\subseteq L_{sp}^{p-1}(\rn)$.
\end{definition}

\begin{definition}
  Let $\Omega$ be an open subset of $\rn$. For any $k\in\zz_+$ and $\sz\in[0,1]$,
denote by $C^{k,\,\sz}(\Omega)$ the \emph{H\"{o}lder
space} consisting of all functions $f:\ \Omega\to\rr$ such that $f$ has  continuous partial derivatives up to order $k$ with
$$
\|f\|_{C^{k,\,\sz}(\Omega)}:=\sum_{\{\nu\in\zz_+^n:\,|\nu|\le k\}}\|D^\nu f\|_{L^\fz(\Omega)}+\sum_{\{\nu\in\zz_+^n:\,|\nu|= k\}}\|D^\nu f\|_{C^\sz(\Omega)}<\infty,
$$
where for a continuous function $g:\ \Omega\to\rr$ define
$$\|g\|_{C^\sz(\Omega)}:=\sup_{x,\,y\in \Omega,\,x\neq y}\frac {|g(x)-g(y)|}{|x-y|^{\sz}}.$$
We write $C^{k,\,0}(\Omega)$ simply as $C^{k}(\Omega)$.
Also, we  write $C^{0,\,\sigma}(\Omega)$ simply as $C^{\sigma}(\Omega)$.
\end{definition}

The following estimate follows easily from an application of the mean value theorem to the function $h(t):=|t|^{p-2}t$ on $\rr$ (see, for instance, \cite[(57)]{ChenLi2018AdvMath}).

\begin{lemma}
  Let $p\in[1,\infty)$. Then there exists a positive constant $C$ such that for any $A,B\in\rn$,
  \begin{align}\label{eq-ChenLi57}
  \left||A+B|^{p-2}(A+B)-|A|^{p-2}A\right|\le C (|A|+|B|)^{p-2} |B|.
  \end{align}
\end{lemma}

Now, we consider the behaviour  of $(-\Delta)^s_p u$ when the function $u$ has some regularity
and integrability properties.

\begin{proposition}\label{prop-ChenLi}
 Let $s\in(0,1)$, $p\in(1,\infty)$  and  $\sigma\in (0,1]$  such that
 $$\sigma>1-(1-s)p.$$
If $u\in L_{sp}^{p-1}(\rn)\cap C^{1,\,\sigma}(B(x_0,\delta_0))$ for some $\delta_0\in(0,\infty)$, then
	\begin{align}\label{eq-finite}
(-\Delta)^s_pu(x_0)
=\mathrm{p.v.}\int_{\rn}\frac{|u(x_0)-u(y)|^{p-2}(u(x_0)-u(y))}{|x_0-y|^{n+sp}}\,dy<\infty,
\end{align}
and the function $(-\Delta)^s_pu$ is continuous at the  point $x_0\in\rn$.
 \end{proposition}

\begin{proof}
Note that \eqref{eq-finite} has essentially been proven in \cite[Lemma~5.2]{ChenLi2018AdvMath}.  To show that $(-\Delta)^s_pu$ is continuous at $x_0$, we fix $x\in\rn$ such that $|x-x_0|<\delta_0/4$.   To simplify the notation, upon setting
\begin{align}\label{eq-h}
h(t):=|t|^{p-2}t\quad \text{for all}\ \, t\in\rr,
\end{align}
we write
\begin{align*}
 &\mathrm{p.v.}\int_{\rn}\frac{|u(x)-u(y)|^{p-2}(u(x)-u(y))}{|x-y|^{n+sp}}\,dy= \mathrm{p.v.}\int_{\rn}\frac{h\big(u(x)-u(y)\big)}{|x-y|^{n+sp}}\,dy.
 \end{align*}

For any $\eta\in(0,\,\delta_0/4)$, by symmetry, we have
\begin{align*}
&\mathrm{p.v.}\int_{|y-x|<\eta}\frac{h\big((x-y)\cdot\nabla u(x)\big)}{|x-y|^{n+sp}}\,dy\\
&\quad=\mathrm{p.v.}\int_{|y-x|<\eta}\frac{|(x-y)\cdot\nabla u(x)|^{p-2}((x-y)\cdot\nabla u(x))}{|x-y|^{n+sp}}\,dy=0,
\end{align*}
which induces
\begin{align*}
 &\mathrm{p.v.}\int_{|y-x|<\eta}\frac{h\big(u(x)-u(y)\big)}{|x-y|^{n+sp}}\,dy\\
 &\quad = \mathrm{p.v.}\int_{|y-x|<\eta}\frac{h\big(u(x)-u(y)\big)-h\big((x-y)\cdot\nabla u(x)\big)}{|x-y|^{n+sp}}\,dy.
\end{align*}
Since $u\in C^{1,\,\sigma}(B(x_0, \delta_0))$, it follows from the Taylor expansion that
for all $x,y\in B(x_0, \delta_0)$,
\begin{align*}
  u(y)-u(x)=(y-x)\cdot\nabla u(x)+O(|x-y|^{1+\sigma}).
\end{align*}
So, if $|y-x|<\eta<\delta_0/4$ and $|x-x_0|<\delta_0/4$ (which implies $|y-x|<\delta_0/2$),  then we apply \eqref{eq-ChenLi57} with
$$A:=(x-y)\cdot\nabla u(x)$$
and
$$B:=u(x)-u(y)-(x-y)\cdot\nabla u(x)=O(|x-y|^{1+\sigma}),$$
which leads to
\begin{align*}
&\left|h\big(u(x)-u(y)\big)-h\big((x-y)\cdot\nabla u(x)\big)\right|\\
&\quad=|h(B+A)-h(A)|\\
&\quad\le C (|A|+|B|)^{p-2} |B|\\
  &\quad\le C'\left||(x-y)\cdot\nabla u(x)|+|x-y|^{1+\sigma}\right|^{p-2}|x-y|^{1+\sigma}\\
  &\quad\le C'' |x-y|^{p+\sigma-1},
\end{align*}
where the constants $C,C',C''$ are independent of $x,y$, but may  depend on the harmless constants $\|u\|_{C^{1,\sigma}(B(x_0,\,\delta_0/2))}$ and $\|\nabla u\|_{L^\infty(B(x_0,\,\delta_0/2))}$.
This, combined with the assumption $(1-s)p+\sigma-1>0$, further derives
\begin{align*}
 \left|\mathrm{p.v.}\int_{|y-x|<\eta}\frac{h\big(u(x)-u(y)\big)}{|x-y|^{n+sp}}\,dy\right|
 &\le C''\int_{|y-x|<\eta}|x-y|^{(1-s)p+\sigma-1-n}\,dy\\
 & = C''\omega_{n-1}\int_0^\eta \rho^{(1-s)p+\sigma-2}\,d\rho\\
 &=C''\omega_{n-1}((1-s)p+\sigma-1)^{-1}\eta^{(1-s)p+\sigma-1}.
 \end{align*}
 So, for any given $\epsilon>0$, we may choose some $\eta\in (0, \delta_0/4)$ sufficiently small to achieve that
\begin{align}\label{eq-eta1}
 \left|\mathrm{p.v.}\int_{|y-x|<\eta}\frac{h\big(u(x)-u(y)\big)}{|x-y|^{n+sp}}\,dy\right|<\epsilon
\end{align}
uniformly in $x\in B(x_0,\,\delta_0/4)$.

Fix such a number $\eta$ that ensures \eqref{eq-eta1}. Suppose now $|x-x_0|<\eta/2$.
By  the continuity of $u$ on $B(x_0, \delta_0)$, we  have
\begin{align*}
\left|h\big(u(x)-u(y)\big)\right| \le (|u(x)|+|u(y)|)^{p-1} \le (\|u\|_{L^\infty(B(x_0,\,\delta_0/2))}+|u(y)|)^{p-1}.
\end{align*}
If $|y-x|\ge\eta$, then it follows from $|x-x_0|<\eta/2<1$ that
\begin{align*}
\frac{1+|y|}{1+|x_0|}\le 1+|y-x_0|\le 1+|x_0-x|+|x-y|<2+|x-y| \le \left(\frac 4{\eta}+1\right)|x-y|
\end{align*}
and, hence,
$$
\frac{\left|h\big(u(x)-u(y)\big)\right|}{|x-y|^{n+sp}} \mathbf 1_{B(x,\,\eta)^c}(y)\le C \frac{1+|u(y)|^{p-1}}{1+|y|^{n+sp}},
$$
with $C$ being a positive constant depending only on $\eta, s, p, n$ and $\|u\|_{L^\infty(B(x_0,\,\delta_0/2))}$. This, together with \eqref{eq-Lsp}, the continuity of $u$, and the Lebesgue dominated convergence theorem, further induces
\begin{align*}
&\lim_{x\to x_0}\int_{|y-x|\ge\eta}\frac{h\big(u(x)-u(y)\big)}{|x-y|^{n+sp}}\,dy
= \int_{|y-x_0|\ge \eta}\frac{h\big(u(x_0)-u(y)\big)}{|x_0-y|^{n+sp}}\,dy.
\end{align*}
Due to this last convergence, for any $\epsilon>0$ we can find some $\delta\in (0, \eta/2)$ such that when $|x-x_0|<\delta$,
\begin{align}\label{eq-Iconvg}
\left|\int_{|y-x|\ge\eta}\frac{h\big(u(x)-u(y)\big)}{|x-y|^{n+sp}}\,dy
- \int_{|y-x_0|\ge \eta}\frac{h\big(u(x_0)-u(y)\big)}{|x_0-y|^{n+sp}}\,dy\right|<\epsilon.
\end{align}

Finally, for any $\epsilon>0$, combining \eqref{eq-eta1} and \eqref{eq-Iconvg} yields that there exists $\delta\in (0,1)$ such that for all $x\in\rn$ satisfying $|x-x_0|<\delta$,
\begin{align*}
  &|(-\Delta)^s_pu(x)-(-\Delta)^s_pu(x_0)|\\
  &\quad =\left|\mathrm{p.v.}\int_{\rn}\frac{h\big(u(x)-u(y)\big)}{|x-y|^{n+sp}}\,dy
  -\mathrm{p.v.}\int_{\rn}\frac{h\big(u(x_0)-u(y)\big)}{|x_0-y|^{n+sp}}\,dy\right|\\
  &\quad\le \left|\mathrm{p.v.}\int_{|y-x|<\eta}\frac{h\big(u(x)-u(y)\big)}{|x-y|^{n+sp}}\,dy\right|\\
  &\qquad +\left|\mathrm{p.v.}\int_{|y-x_0|<\eta}\frac{h\big(u(x_0)-u(y)\big)}{|x_0-y|^{n+sp}}\,dy\right|\\
  &\qquad
  +\left|\int_{|y-x|\ge\eta}\frac{h\big(u(x)-u(y)\big)}{|x-y|^{n+sp}}\,dy
- \int_{|y-x_0|\ge \eta}\frac{h\big(u(x_0)-u(y)\big)}{|x_0-y|^{n+sp}}\,dy\right|\\
  &\quad<3\epsilon.
\end{align*}
This proves the continuity of $(-\Delta)^s_pu$ at the point $x_0\in\rn$.
\end{proof}

\begin{remark}
As a consequence of Proposition \ref{prop-ChenLi}, we see that for the function $u_0$ in \eqref{eq-u0}, the function $(-\Delta)^s_p u_0$ is finite and continuous on $\rn\setminus\{0\}$.
\end{remark}

\subsection{A technical proposition}\label{ss2.2}

The following two simple estimates will be used throughout the paper.

\begin{lemma}\label{lem1}
Let $n\in\nn$  and $\alpha\in (0,n)$. Then for any compact set $K\subseteq\rn$, $R\in(0,\infty)$ and $x\in\rn$,
\begin{align}\label{eq-basic}
  \int_K  |x-y|^{\alpha-n}\,dy\le \alpha^{-1} \omega_{n-1} \left(\frac{|K|}{\nu_n}\right)^{\alpha/n}.
\end{align}
and
\begin{align}\label{eq-basic2}
  \int_{|x-y|\ge R}  |x-y|^{-\alpha-n}\,dy\le \alpha^{-1} \omega_{n-1} R^{-\alpha},
\end{align}
where $\nu_n$ and $\omega_{n-1}$ denote the Lebesgue measure and the surface measure of the unit ball in $\rn$, respectively.
\end{lemma}

\begin{proof}
To obtain \eqref{eq-basic}, for any $x\in\rn$, applying  the Fubini theorem yields
\begin{align*}
\int_{K}  |x-y|^{\alpha-n}\,dy
&= (n-\alpha)\int_K \int_{|x-y|}^\infty   t^{\alpha-n-1}\,dt\,dy\\
&=(n-\alpha) \int_0^\infty \left(\int_{\{x\in K:\, |x-y|\le t\}} \,dy\right) t^{\alpha-n-1}\,dt\notag\\
&\le (n-\alpha) \int_0^\infty  \min\left\{\nu_nt^n,\, |K|\right\}   t^{\alpha-n-1}\,dt\notag\\
&= (n-\alpha)\left(\nu_n
\int_0^{(|K|/\nu_n)^{1/n}}  t^{\alpha-1}\,dt
+ |K|\int_{(|K|/\nu_n)^{1/n}}^\infty  t^{\alpha-n-1}\,dt\right)\\
&= n\alpha^{-1} \nu_n \left(\frac{|K|}{\nu_n}\right)^{\alpha/n}.\notag
\end{align*}
Moreover, observe that \eqref{eq-basic2} follows simply from
\begin{align*}
\int_{|x-y|\ge R}  |x-y|^{-\alpha-n}\,dy
&=\omega_{n-1}\int_R^\infty \rho^{-\alpha-1}\,d\rho =\omega_{n-1}\alpha^{-1}R^{-\alpha}.
\end{align*}
\end{proof}

The main goal of this section is to establish the following technical estimate.

\begin{proposition}\label{prop1}
  Let $s\in(0,1)$, $p\in (1, n/s]$ and $u_0$ be as in \eqref{eq-u0}. Then
  $$
\cj:= \int_{\rn} \int_{\rn}\frac{|u_0(x)-u_0(y)|^{p-1}}{|x-y|^{n+sp}}\left||x-{\bf e}_1|^{s-n}-|y-{\bf e}_1|^{s-n}\right|\,dy\,dx<\infty.
$$
Consequently, the value $c_{n,s,p}^\ast$ in \eqref{eq-cnsp} is a finite number.
\end{proposition}
\begin{proof}
Due to symmetry, we consider only the integrand of $\cj$ over   $$\left\{(x,y)\in\rn\times\rn:\,|x|\le |y|\right\}.$$
Set
$$
\begin{cases}
  \Omega_1:=   \{(x,y)\in\rn\times\rn:\,|x|\le |y|,\, |x-y|\le |x|/4,\, |x-y|\le |x-{\bf e}_1|/2\};\\
  \Omega_2:=   \{(x,y)\in\rn\times\rn:\,|x|\le |y|,\, |x-y|\le |x|/4,\, |x-y|> |x-{\bf e}_1|/2\};\\
  \Omega_3:=   \{(x,y)\in\rn\times\rn:\,|x|\le |y|,\, |x-y|> |x|/4,\, |x-y|\le |x-{\bf e}_1|/9\};\\
  \Omega_4:=   \{(x,y)\in\rn\times\rn:\,|x|\le |y|,\, |x-y|> |x|/4,\, |x-y|> |x-{\bf e}_1|/9\}.
\end{cases}
$$
For $i\in\{1,2,3,4\}$, denote by $\cj_i$ the corresponding integral
of the integrand of $\cj$ over  $\Omega_i$.
It suffices to show that each $\cj_i$ is finite.

\medskip

{\bf  Step 1:\, verification of $\cj_1<\infty$.\,}
If $sp<n$, then  the mean value theorem implies that there exists some $\theta\in (0,1)$ such that for any $t_1,t_2\in (0,\infty)$,
$$
\left|t_1^{\frac{sp-n}{p-1}}-t_2^{\frac{sp-n}{p-1}}\right|
=\frac{sp-n}{p-1}|t_1-t_2| \left|t_1+\theta(t_2-t_1)\right|^{\frac{sp-n}{p-1}-1}
\le \frac{sp-n}{p-1}|t_1-t_2| \min\{t_1,\,t_2\}^{\frac{sp-n}{p-1}-1}.
$$
Under the case $sp=n$, in a similar manner we
obtain that for any $t_1,t_2\in (0,\infty)$,
$$
\left|\ln t_1-\ln t_2\right|
=|t_1-t_2| \left|t_1+\theta(t_2-t_1)\right|^{-1}
\le |t_1-t_2| \min\{t_1,\,t_2\}^{-1}.
$$
Note that $|x-y|\le |x|/4$ implies
 $|x|/2\le |y|\le 2|x|.$
Thus, when  $|x-y|\le |x|/4$, no matter $sp<n$ or $sp=n$, we always have
\begin{align}\label{eq-u0xy}
 |u_0(x)-u_0(y)|\ls \big||x|-|y|\big|\, \min\{|x|,\, |y|\}^{\frac{sp-n}{p-1}-1}
 \ls |x-y|\, |x|^{\frac{sp-n}{p-1}-1}.
\end{align}
Similarly,   the mean value theorem also implies that when $|x-y|\le |x-{\bf e}_1|/2$,
\begin{align}\label{eq-Isxy}
 \left||x-{\bf e}_1|^{s-n}-|y-{\bf e}_1|^{s-n}\right|
 \ls |x-y|\, |x-{\bf e}_1|^{s-n-1}.
\end{align}
An application of \eqref{eq-u0xy} and \eqref{eq-Isxy} gives us that
\begin{align*}
  \cj_1 &\ls \iint_{\Omega_1} \frac{|x-y|^{p-1} |x|^{sp-n-(p-1)}}{|x-y|^{n+sp}} |x-y||x-{\bf e}_1|^{s-n-1}\,dy\,dx \\
  & \ls  \iint_{|x-y|\le\min\{|x|/4,\,|x-{\bf e}_1|/2\}} |x-y|^{(1-s)p-n} |x|^{sp-n-p+1} |x-{\bf e}_1|^{s-n-1}\,dy\,dx.
\end{align*}
Integrating in the variable $y$ and using \eqref{eq-basic}, we obtain
$$
\int_{|x-y|\le\min\{|x|/4,\,|x-{\bf e}_1|/2\}} |x-y|^{(1-s)p-n} \,dy\ls \left( \min\left\{\frac{|x|}4,\,\frac{|x-{\bf e}_1|}2\right\}\right)^{(1-s)p}.
$$
Consequently, we have
\begin{align*}
  \cj_1 &\ls  \int_{\rn}\left( \min\left\{\frac{|x|}4,\,\frac{|x-{\bf e}_1|}2\right\}\right)^{(1-s)p}|x|^{sp-n-p+1} |x-{\bf e}_1|^{s-n-1}\,dx \\
   &\simeq  \left(\int_{|x|\le 1/2}
   +
   \int_{1/2<|x|<2}
   +\int_{|x|\ge 2}\right) \cdots\\
   & \ls \int_{|x|\le 1/2} |x|^{1-n}\,dx
   + \int_{1/2<|x|<2} |x-{\bf e}_1|^{(1-s)(p-1)-n}\,dx
   + \int_{|x|\ge 2} |x|^{s-2n}\,dx\\
   &\ls 1,
\end{align*}
as desired.

\medskip

{\bf  Step 2:\, verification of $\cj_2<\infty$.\,}
For any $(x,y)\in \Omega_2$, we still have $|x-y|\le |x|/4$, so that
\eqref{eq-u0xy} remains valid, which implies
\begin{align*}
  \cj_2 & \ls  \iint_{\Omega_2} |x-y|^{(1-s)p-n-1} |x|^{sp-n-(p-1)}
  \left(|x-{\bf e}_1|^{s-n}+|y-{\bf e}_1|^{s-n} \right)\,dy\,dx.
\end{align*}

If $(x,y)\in\Omega_2$, then
$|x-{\bf e}_1|/2< |x-y|\le |x|/4$, which implies
$$
|x|\le |x-{\bf e}_1|+|{\bf e}_1|<|x|/4+1
$$
and, hence,
$|x|\le 4/3.$
Meanwhile, from $|x-{\bf e}_1|/2<  |x|/4$, it also follows that
$$
1=|{\bf e}_1|\le |{\bf e}_1-x|+|x|<3|x|/2.
$$
Combining these facts yields  $1/2< |x|<2$.  As was mentioned in {\bf Step 1}, the condition $|x-y|\le |x|/4$ implies $|x|/2\le |y|\le 2|x|$, so that $1/4<|y|< 4$. By $|x-y|>|x-{\bf e}_1|/2$, there is
$$
|y-{\bf e}_1|\le |y-x|+|x-{\bf e}_1|< 3|x-y|.
$$
Altogether, we see that any $(x,y)\in\Omega_2$ satisfies
$$
1/4<|x|<4,\quad 1/4<  |y|< 4\quad \text{and}\quad 3|x-y|> \max\{|x-{\bf e}_1|,\, |y-{\bf e}_1| \}.
$$
So we can continue with the estimate of $\cj_2$ with
\begin{align*}
  \cj_2 & \ls   \iint_{\gfz{1/4< |x|<4,\ 1/4<|y|< 4}{3|x-y|>\max\{|x-{\bf e}_1|,\ |y-{\bf e}_1| \}}} |x-y|^{(1-s)p-n-1}
  \left(|x-{\bf e}_1|^{s-n}+|y-{\bf e}_1|^{s-n} \right)\,dy\,dx.
\end{align*}
Due to symmetry in $x$ and $y$, the estimate of $\cj_2$ falls into the following:
$$
\cj_2\ls \iint_{\gfz{1/4< |x|<4,\ 1/4<|y|< 4}{|x-y|>|x-{\bf e}_1|/3}} |x-y|^{(1-s)p-n-1}
 |x-{\bf e}_1|^{s-n}\,dy\,dx
$$
Starting from here, we will show that $\cj_2\ls 1$ by considering the cases $(1-s)p>1$ and $(1-s)p\le 1$, respectively.

\begin{itemize}
  \item[--] {\it The case $(1-s)p>1$.} Since $(1-s)p-1>0$, it follows directly from \eqref{eq-basic} that
  $$
 \int_{1/4<|y|< 4} |x-y|^{(1-s)p-n-1}\,dy \ls 1
  $$
  holds uniformly in $1/4< |x|<4$,
  which further implies
  $$
  \cj_2\ls \int_{|x|<4  } |x-{\bf e}_1|^{s-n}\,dx\ls 1
  $$
  by using \eqref{eq-basic} again.

  \item[--] {\it The case $(1-s)p\le 1$.}  In this case,  we take a small value $\sigma$ such that $0<\sigma< (1-s)(p-1).$
Then,
$$s+[(1-s)p-\sigma-1]= (1-s)(p-1)-\sigma>0.$$
Note that $(1-s)p-\sigma-1\le -\sigma <0$ and $|x-y|>|x-{\bf e}_1|/3$, we have
$$
|x-y|^{(1-s)p-n-1}= |x-y|^{\sigma-n}|x-y|^{(1-s)p-\sigma-1}\ls |x-y|^{\sigma-n}|x-{\bf e}_1|^{(1-s)p-\sigma-1}.
$$
So, applying \eqref{eq-basic}  twice yields
\begin{align*}
  \cj_2 &\ls \int_{1/4< |x|<4}\int_{1/4<|y|< 4}  |x-y|^{\sigma-n} |x-{\bf e}_1|^{s+[(1-s)p-\sigma-1]-n}\,dy\,dx\\
  &\ls  \int_{1/4< |x|<4}|x-{\bf e}_1|^{(1-s)(p-1)-\sigma-n}\,dx\\
  &\ls 1.
\end{align*}
\end{itemize}

\medskip

{\bf  Step 3:\, verification of $\cj_3<\infty$.\,}
For any $(x,y)\in\Omega_3$,  we claim that
\begin{align}\label{eq-xy-range3}
|x|< \frac 45, \quad  |y|< 1\quad \text{and}\quad |x-y|<\frac 15< |x-{\bf e}_1|< 2.
\end{align}
Indeed, by the fact $|x|/4<|x-y|\le |x-{\bf e}_1|/9$, we have
$$
|x|<\frac {4|{\bf e}_1-x|} 9\le \frac {4(|{\bf e}_1|+|x|)}9 =\frac {4(1+|x|)}9 ,
$$
which implies
$$
\frac 54 |x|< \frac 49 \ \ \Leftrightarrow\ \ |x|< \frac 45.
$$
This last fact further implies
$$
\frac 15\le |{\bf e}_1|-|x|\le  |{\bf e}_1-x|\le |{\bf e}_1|+|x|< 1+\frac45<2
$$
and
$$|y-x|\le \frac{|x-{\bf e}_1|}9 \le \frac{|x|+1} 9< \frac 15$$
and, hence,
$$|y|\le |y-x|+|x|<\frac 1 5+\frac 45= 1.$$
Altogether, we obtain \eqref{eq-xy-range3}.

Note that \eqref{eq-Isxy} remains to be true whenever $|x-y|\le |x-{\bf e}_1|/9$.
Thus,
\begin{align*}
  \cj_3 & \ls \iint_{\Omega_3} \frac{|u_0(x)-u_0(y)|^{p-1}}{|x-y|^{n+sp}} |x-y||x-{\bf e}_1|^{s-n-1}\,dy\,dx\\
  &\simeq  \iint_{\Omega_3}|u_0(x)-u_0(y)|^{p-1} |x-y|^{1-sp-n}\,dy\,dx,
\end{align*}
where in the second step we used $|x-{\bf e}_1|\simeq  1$ (see \eqref{eq-xy-range3}).
Next, instead of \eqref{eq-u0xy}, we will apply
\begin{align}\label{eq-u0x+y}
|u_0(x)-u_0(y)|
\le \begin{cases}
      |x|^{\frac{sp-n}{p-1}}\quad & \textup{when}\ sp<n \ \textup{and}\ |x|\le |y|;\\
     \ln \frac{|y|}{|x|}\quad & \textup{when}\ sp=n \ \textup{and}\ |x|\le |y|.
    \end{cases}
\end{align}
With \eqref{eq-xy-range3} and \eqref{eq-u0x+y}, let us show  $\cj_3\ls 1$ by considering the cases $sp<n$ and $sp=n$, respectively.

\begin{itemize}
  \item[--] {\it The case $sp<n$.\,}
  In this case, from \eqref{eq-xy-range3} and the first estimate of \eqref{eq-u0x+y}, it follows that
  \begin{align*}
  \cj_3
  &\ls \int_{|x|<4/5} \int_{|x|/4\le |x-y|<1/5} |x|^{sp-n}|x-y|^{1-sp-n}\,dy\,dx\\
  &\ls
  \begin{cases}
    \dint_{|x|<4/5} |x|^{sp-n}\,dx \quad &\text{as}\ sp<1;\\
    \dint_{|x|<4/5} |x|^{sp-n} \ln \frac4{5|x|} \,dx \quad &\text{as}\ sp=1;\\
    \dint_{|x|<4/5} |x|^{sp-n} |x|^{1-sp}\,dx\quad &\text{as}\ sp>1;
  \end{cases}\\
  &\ls 1,
  \end{align*}
as desired.

  \item[--] {\it The case $sp=n$.\,} If $|x|\le |y|\le 2|x|$, then
  $$|u_0(x)-u_0(y)|=\ln \frac{|y|}{|x|}\le \ln 2,$$
  so that
  \begin{align}\label{eq1-j3}
   &\iint_{\{(x,y)\in\Omega_3:\ |y|\le 2|x|\}}|u_0(x)-u_0(y)|^{p-1} |x-y|^{1-sp-n}\,dy\,dx \\
    &\quad\ls \int_{|x|<4/5} \int_{|x|/4\le |x-y|<1/5} |x-y|^{1-2n}\,dy\,dx \notag\\
    &\quad\ls
  \begin{cases}
       \dint_{|x|<4/5}  \ln \frac4{5|x|} \,dx \quad &\text{as}\ n=1;\\
    \dint_{|x|<4/5}  |x|^{1- n}\,dx\quad &\text{as}\ n>1;
  \end{cases} \notag\\
  &\quad\ls 1.  \notag
  \end{align}

 If $|y|> 2|x|$, then $|y|/2< |x-y|< 2|y|$.
 Invoking  \eqref{eq-xy-range3}, we then write
 \begin{align}\label{eq2-j3}
   &\iint_{\{(x,y)\in\Omega_3:\ |y|> 2|x|\}}|u_0(x)-u_0(y)|^{p-1} |x-y|^{1-sp-n}\,dy\,dx \\
    &\quad\ls \int_{|x|<4/5} \int_{2|x|<|y|<1} \left(\ln\frac{|y|}{|x|}\right)^{p-1}|y|^{1-2n}\,dy\,dx \notag\\
    &\quad\simeq  \int_0^{4/5}  \left(\int_{2t}^1 \left(\ln\frac{\rho}{t}\right)^{p-1}\rho^{-n}t^{n-1}\,d\rho\right)\,dt \notag\\
    &\quad\simeq  \int_0^{4/5}  \left( \int_{2}^{\frac 1t} \left(\ln\tau\right)^{p-1}\tau^{-n}\,d\tau\right)\,dt \notag\\
    &\quad \ls
  \begin{cases}
      \dint_0^{4/5}  \left(\ln\frac 1t\right)^{p} \,dt \quad &\text{as}\ n=1;\\
     \dint_0^{4/5} \,dt\quad &\text{as}\ n>1;
  \end{cases} \notag\\
  &\quad\ls 1.  \notag
  \end{align}
  Combining \eqref{eq1-j3} and \eqref{eq2-j3} yields that $\cj_3\ls 1$ under the case $sp=n$.
\end{itemize}

\medskip

{\bf  Step 4:\, verification of $\cj_4<\infty$ under $sp<n$.\,}
By  \eqref{eq-u0x+y}, we write
\begin{align*}
  \cj_4 &\le \iint_{\Omega_4} \frac{|x|^{sp-n}}{|x-y|^{n+sp}} \left(|x-{\bf e}_1|^{s-n}+|y-{\bf e}_1|^{s-n}\right)\,dy\,dx\\
  &= \iint_{\Omega_4} \frac{|x|^{sp-n}}{|x-y|^{n+sp}} |x-{\bf e}_1|^{s-n}\,dy\,dx
  +\iint_{\Omega_4} \frac{|x|^{sp-n}}{|x-y|^{n+sp}} |y-{\bf e}_1|^{s-n}\,dy\,dx\\
  &=:\cj_{41}+\cj_{42}.
\end{align*}

\begin{itemize}
  \item[--] {\it Estimation of $\cj_{41}$.} Observe that \eqref{eq-basic2} implies
  $$
  \int_{|x-y|>\max\{|x|/4,\, |x-{\bf e}_1|/9\}} \frac1{|x-y|^{n+sp}}\,dy\ls \left(\max\left\{|x|,\, |x-{\bf e}_1|\right\}\right)^{-sp},
  $$
which further induces
  \begin{align*}
    \cj_{41} &\ls \int_\rn  \left(\max\left\{|x|,\, |x-{\bf e}_1|\right\}\right)^{-sp}  |x|^{sp-n} |x-{\bf e}_1|^{s-n}\,dx\\
    &\ls \int_{|x|<1/2} |x|^{sp-n} \,dx+\int_{1/2\le |x|\le 2} |x-{\bf e}_1|^{s-n}\,dx
    +\int_{|x|\ge 2} |x|^{s-2n}\,dx  \\
    &\ls 1.
  \end{align*}

  \item[--] {\it Estimation of $\cj_{42}$.}
 On the one hand, it is obvious that
 \begin{align}\label{eq1-j42}\iint_{\{(x,y)\in\Omega_4:\ |y-{\bf e}_1|>|x-{\bf e}_1|/2\}} \frac{|x|^{sp-n}}{|x-y|^{n+sp}} |y-{\bf e}_1|^{s-n}\,dy\,dx\le 2^{n-s}\cj_{41}\ls 1.
 \end{align}
On the other hand, for any $(x,y)\in\Omega_4$, we have
$$
|y|\le |y-x|+|x|< 5|x-y|
$$
and
$$
|y-{\bf e}_1|\le |y-x|+|x-{\bf e}_1|< 10|y-x|,
$$
which implies
\begin{align}\label{eq2-j42}
 &\iint_{\{(x,y)\in\Omega_4:\ |y|\le 4|x|\}} \frac{|x|^{sp-n}}{|x-y|^{n+sp}} |y-{\bf e}_1|^{s-n}\,dy\,dx\\
 &\quad\le 4^{n-sp}\int_{y\in\rn}\left(\int_{|x-y|>\max\{|y|/5,\ |y-{\bf e}_1|/10\}} \frac{1}{|x-y|^{n+sp}} \,dx \right)|y|^{sp-n}|y-{\bf e}_1|^{s-n}\,dy\notag\\
 &\quad\ls 1\notag
\end{align}
by terms of a similar estimate as that of $\cj_{41}$.

Under the case   $|y-{\bf e}_1|<|x-{\bf e}_1|/2$ and $|y|>4|x|$, we find that
$$|y|/2<|x-y|<2|y|$$
and
$$
|y|\le |y-{\bf e}_1|+|{\bf e}_1| <\frac{|x-{\bf e}_1|}2+1\le \frac{|x|+1}2+1\le \frac{|y|+1}2+1,
$$
where the latter implies  $|y|<3$ and, hence, $|x|<3/4$. In particular, if $|y|<1/2$, then
$$|y|\ge |{\bf e}_1|-|{\bf e}_1-y|>1-\frac{|x-{\bf e}_1|}2> 1- \frac{|x|+1}2\ge \frac12-\frac{|x|}2>\frac 12-\frac 38=\frac 18.$$
So, if  $|y-{\bf e}_1|<|x-{\bf e}_1|/2$ and $|y|>4|x|$, then we always have $1/8<|y|<3$.
Consequently,
\begin{align}\label{eq3-j42}
 &\iint_{\{(x,y)\in\Omega_4:\ |y-{\bf e}_1|<|x-{\bf e}_1|/2,\ |y|> 4|x|\}} \frac{|x|^{sp-n}}{|x-y|^{n+sp}} |y-{\bf e}_1|^{s-n}\,dy\,dx\\
 &\quad\simeq \iint_{\{(x,y)\in\Omega_4:\ |y-{\bf e}_1|<|x-{\bf e}_1|/2,\ |y|> 4|x|\}} \frac{|x|^{sp-n}}{|y|^{n+sp}} |y-{\bf e}_1|^{s-n}\,dy\,dx\notag\\
 &\quad\ls\int_{1/8<|y|<3} \left(\int_{|x|<|y|/4}\frac{|x|^{sp-n}}{|y|^{n+sp}} |y-{\bf e}_1|^{s-n}\,dx\right)\,dy\notag\\
&\quad\ls\int_{1/8<|y|<3}|y|^{-n}|y-{\bf e}_1|^{s-n}\,dy\notag\\
 &\quad\ls 1.\notag
\end{align}
A combination of \eqref{eq1-j42}-\eqref{eq2-j42}-\eqref{eq3-j42} yields $\cj_{42}\ls 1$.
  \end{itemize}

From the facts $\cj_{41}\ls 1$ and $\cj_{42}\ls 1$, it follows that $\cj_{4}\ls 1$ under the case $sp<n$.

\medskip

{\bf  Step 5:\, verification of $\cj_4<\infty$ under $sp=n$.\,}
In this case, applying   \eqref{eq-u0x+y} gives
\begin{align*}
  \cj_4 &\le \iint_{\Omega_4} \frac{(\ln \frac{|y|}{|x|})^{p-1}}{|x-y|^{n+sp}} \left(|x-{\bf e}_1|^{s-n}+|y-{\bf e}_1|^{s-n}\right)\,dy\,dx.
\end{align*}
If $|y|< 4|x|$, then $0< \ln \frac{|y|}{|x|}<\ln 4=(\ln 4) |x|^{\frac{sp-n}{p-1}}$, so the arguments in {\bf Step 4} also implies
\begin{align*}
\cj_{43}&:=\iint_{\{(x,y)\in\Omega_4:\, |y|< 4|x|\}} \frac{(\ln \frac{|y|}{|x|})^{p-1}}{|x-y|^{n+sp}} \left(|x-{\bf e}_1|^{s-n}+|y-{\bf e}_1|^{s-n}\right)\,dy\,dx\\
&
\le (\ln 4)^{p-1}
\iint_{\Omega_4} \frac{|x|^{sp-n}}{|x-y|^{n+sp}} \left(|x-{\bf e}_1|^{s-n}+|y-{\bf e}_1|^{s-n}\right)\,dy\,dx\\
& \ls 1.
\end{align*}
If $|y|\ge 4|x|$, then $|y|/2\le  |y-x|\le2|y|$. Moreover, from $|y|\ge 4|x|$ and the condition $|x-y|> |x-{\bf e}_1|/9$ that was given in the definition of $\Omega_4$, it follows that
    $$
  1=|{\bf e}_1|\le |{\bf e}_1 -x|+|x|<9|x-y|+|x|<9(|x|+|y|)+|x|<12|y|.
  $$
Thus,
\begin{align*}
\cj_{44}&:=\iint_{\{(x,y)\in\Omega_4:\, |y|\ge 4|x|\}} \frac{(\ln \frac{|y|}{|x|})^{p-1}}{|x-y|^{n+sp}} \left(|x-{\bf e}_1|^{s-n}+|y-{\bf e}_1|^{s-n}\right)\,dy\,dx\\
&\simeq
\iint_{\{(x,y)\in\Omega_4:\, |y|\ge \max\{4|x|, \frac1{12}\}\}} \frac{(\ln \frac{|y|}{|x|})^{p-1}}{|y|^{n+sp}} \left(|x-{\bf e}_1|^{s-n}+|y-{\bf e}_1|^{s-n}\right)\,dy\,dx.
\end{align*}
Next, we will  estimate  $\cj_{44}$  by splitting the integral domain into the following three parts.

\begin{itemize}
  \item[--] {\it Part 1:\, $|x|\le \frac 1{48}$.\,}
If $|x|\le \frac 1{48}$, then $|x-{\bf e}_1|\simeq  1$, which implies
\begin{align}\label{eq1-J44}
&
\iint_{\{(x,y)\in\Omega_4:\, |y|\ge \max\{4|x|, \frac1{12}\},\,|x|\le \frac 1{48}\,\}} \frac{(\ln \frac{|y|}{|x|})^{p-1}}{|y|^{n+sp}} \left(|x-{\bf e}_1|^{s-n}+|y-{\bf e}_1|^{s-n}\right)\,dy\,dx\\
&\quad \ls \int_{|y|\ge  \frac1{12}} \left(\int_{|x|\le \frac 1{48}}\left(\ln (24|y|)+\ln \frac{1}{24|x|}\right)^{p-1}\,dx \right)
|y|^{-(n+sp)} \left(1+|y-{\bf e}_1|^{s-n}\right)\,dy\notag\\
&\quad \ls  \int_{|y|\ge  \frac1{12}} \left[ \ln (24|y|)\right]^{p-1}|y|^{-(n+sp)} \left(1+|y-{\bf e}_1|^{s-n}\right)\,dy\notag\\
&\quad \ls \int_{ \frac1{12}\le |y|<2} \left(1+|y-{\bf e}_1|^{s-n}\right)\,dy
+ \int_{|y|\ge 2} \frac{( \ln |y|)^{p-1} }{
|y|^{n+sp}}   \,dy\notag\\
&\quad \ls 1.\notag
\end{align}

\item[--] {\it Part 2:\, $\frac 1{48}<|x|<4$.\,}  If $(x,y)\in\Omega_4$ such that  $|y|\ge \max\{4|x|, \frac1{12}\}$ and $\frac 1{48}<|x|<4$, then $$\ln \frac{|y|}{|x|}\simeq  \ln(24|y|),$$
so we have by Lemma \ref{lem1} that
\begin{align}\label{eq2-J44}
&
\iint_{\{(x,y)\in\Omega_4:\, |y|\ge \max\{4|x|, \frac1{12}\},\,\frac 1{48}<|x|<4\,\}} \frac{(\ln \frac{|y|}{|x|})^{p-1}}{|y|^{n+sp}} \left(|x-{\bf e}_1|^{s-n}+|y-{\bf e}_1|^{s-n}\right)\,dy\,dx\\
&\quad \ls\int_{|y|\ge   \frac1{12}} \frac{\left[ \ln (24|y|)\right]^{p-1} }{
|y|^{n+sp}} \left(\int_{\frac 1{48}<|x|<4}|x-{\bf e}_1|^{s-n}\,dx+|y-{\bf e}_1|^{s-n}\right)\,dy\notag\\
&\quad \ls  \int_{|y|\ge   \frac1{12}} \frac{\left[ \ln (24|y|)\right]^{p-1} }{
|y|^{n+sp}} \left(1+|y-{\bf e}_1|^{s-n}\right)\,dy\notag\\
&\quad \ls \int_{ \frac1{12}\le |y|<2} \left(1+|y-{\bf e}_1|^{s-n}\right)\,dy
+ \int_{|y|\ge 2} \frac{( \ln |y|)^{p-1} }{
|y|^{n+sp}}   \,dy\notag\\
&\quad \ls 1.\notag
\end{align}

\item[--]  {\it Part 3:\, $|x|\ge 4$.\,}  If $(x,y)\in\Omega_4$ such that  $|y|\ge  \max\{4|x|, \frac1{12}\}$ and $|x|\ge 4$, then $|y|\ge 4|x|\ge 16$ and, hence,
\begin{align*}
\left(\ln \frac{|y|}{|x|}\right)^{p-1} \left(|x-{\bf e}_1|^{s-n}+|y-{\bf e}_1|^{s-n}\right)
&\simeq  \left(\ln \frac{|y|}{|x|}\right)^{p-1} \left(|x|^{s-n}+|y|^{s-n}\right)\\
&\ls \left(\ln \frac{|y|}{|x|}\right)^{p-1} |x|^{s-n},
\end{align*}
which further gives
\begin{align}\label{eq3-J44}
&
\iint_{\{(x,y)\in\Omega_4:\, |y|\ge  \max\{4|x|, \frac1{12}\},\,|x|\ge 4\,\}} \frac{(\ln \frac{|y|}{|x|})^{p-1}}{|y|^{n+sp}} \left(|x-{\bf e}_1|^{s-n}+|y-{\bf e}_1|^{s-n}\right)\,dy\,dx\\
&\quad \ls\int_{|y|\ge  16} \int_{|x|\le |y|/4} \left(\ln \frac{|y|}{|x|}\right)^{p-1} |y|^{-(n+sp)}|x|^{s-n}\,dy\,dx\notag\\
&\quad \simeq
\int_{16}^\infty\left( \int_{0}^{t/4}  \left(\ln \frac{t}{\rho}\right)^{p-1} \rho^{s-1}\,d\rho\right)t^{-sp-1}\,dt
\notag\\
&\quad \simeq
\int_{16}^\infty\left( \int_{0}^{1/4} \left(\ln \frac 1\tau \right)^{p-1} \tau^{s-1}\,dt\right)t^{s(1-p)-1}\,dt
\notag\\
&\quad \ls
\int_{16}^\infty t^{s(1-p)-1}\,dt
\notag\\
&\quad \ls 1.\notag
\end{align}
\end{itemize}
Based on the estimates in \eqref{eq1-J44}-\eqref{eq2-J44}-\eqref{eq3-J44}, we directly obtain $\cj_{44}\ls 1$.
This, combined with  $\cj_{43}\le 1$, further induces that $\cj_4\ls 1$ under the case $sp=n$.
\end{proof}

\subsection{Proof of Theorem \ref{thm-solution}}\label{ss2.3}

\begin{proof}[Proof of Theorem \ref{thm-solution}]
To simplify the notation, we set
$$F(u_0)(x,y):=|u_0(x)-u_0(y)|^{p-2}(u_0(x)-u_0(y)).$$
Recall that for $\az\in(0,n)$ the $\az$-th order  {\it Riesz potential} $\mathcal I_\alpha:=(-\Delta)^{-\alpha/2}$ has the integral expression
\begin{equation*}
\mathcal I_\alpha f(x)=c_{n,\alpha}\int_{\mathbb R^n} \frac{f(z)}{|x-z|^{n-\alpha}}\,dz
\end{equation*}
provided that $f$ has sufficiently nice decay at infinity,
where
$$
c_{n,\alpha}:=\frac{\Gamma(\frac{n-\alpha}{2})}{\pi^{\frac n2}2^\alpha \Gamma(\frac\alpha 2)}
$$
and $\Gamma$ is the Gamma function. In particular, for any $\phi\in C_c^\infty(\rn)$,
it is known that as a consequence of the Fourier transform the identity
 \begin{align}\label{eq-phi}
  \phi(x)={\mathcal I}_s\big((-\Delta)^\frac s2 \phi\big)(x)
\end{align}
holds  for all $x\in\rn$ (see \cite[Lemma~2.3]{LiuXiao2021ACHA}). Thus,  for any $x,y\in\rn$,  it is reasonable to write
  \begin{align*}
    \phi(x)-\phi(y)=c_{n,s}\int_{\rn} \left(|x-z|^{s-n}-|y-z|^{s-n}\right) \big((-\Delta)^\frac s2 \phi\big)(z)\,dz.
  \end{align*}
  Consequently, we have
  \begin{align}\label{eq-Fu0}
    &\int_{\rn}\int_{\rn}\frac{|u_0(x)-u_0(y)|^{p-2}(u_0(x)-u_0(y))(\phi(x)-\phi(y))}{|x-y|^{n+sp}}\,dy\,dx\\
    &\quad=\int_{\rn}\int_{\rn}\frac{F(u_0)(x,y)(\phi(x)-\phi(y))}{|x-y|^{n+sp}}\,dy\,dx\notag\\
    &\quad=c_{n,s}\int_{\rn}\int_{\rn}\int_{\rn}\frac{F(u_0)(x,y)\left(|x-z|^{s-n}-|y-z|^{s-n}\right)}{|x-y|^{n+sp}} \big((-\Delta)^\frac s2 \phi\big)(z)\,dz\,dy\,dx.\notag
  \end{align}
 Next, we claim  that
\begin{align}\label{eq-claim}
\int_{\rn}\int_{\rn}\int_{\rn}\left|\frac{F(u_0)(x,y)\left(|x-z|^{s-n}-|y-z|^{s-n}\right)}{|x-y|^{n+sp}}  \big((-\Delta)^\frac s2 \phi\big)(z)\right|\,dz\,dy\,dx<\infty.
\end{align}
Before going further, we will apply Proposition \ref{prop1} to show \eqref{eq-claim}.

Fix $z\neq 0$ and let $r=|z|$. Let $A$ be a rotation matrix in $\rn$ such that $z= rA{\bf e}_1$.
Since $u_0$ is a radial function and of homogeneity order $\frac{sp-n}{p-1}$, it follows that
$$u_0(rAx)-u_0(rAy)=r^{\frac{sp-n}{p-1}}(u_0(x)-u_0(y))$$
and, hence,
\begin{align*}
 F(u_0)(rAx,rAy)
  &=r^{sp-n} F(u_0)(x,y).
\end{align*}
With these facts in hand, we perform a change of variables by replacing
$x$ and
$y$ with $rA x$ and $rA y$, respectively,  thereby leading to
\begin{align}\label{eq-Fu01}
 & \int_{\rn} \int_{\rn}\frac{|F(u_0)(x,y)|\left||x-z|^{s-n}-|y-z|^{s-n}\right|}{|x-y|^{n+sp}}  \,dy\,dx \\
  &\quad =r^{2n}\int_{\rn} \int_{\rn}\frac{|F(u_0)(rA x,rA y)|\left||rA x-rA{\bf e}_1|^{s-n}-|rA y-rA{\bf e}_1|^{s-n}\right|}{|rA x-rA y|^{n+sp}}  \,dy\,dx\notag\\
  &\quad =r^{s-n} \int_{\rn} \int_{\rn}\frac{|F(u_0)(x,y)|\left||x-{\bf e}_1|^{s-n}-|y-{\bf e}_1|^{s-n}\right|}{|x-y|^{n+sp}}\,dy\,dx\notag\\
  &\quad =C(n,s,p)|z|^{s-n},\notag
\end{align}
where in the last step we used Proposition \ref{prop1} to obtain that
$$
\int_{\rn}\int_{\rn}\frac{|F(u_0)(x,y)|\left||x-{\bf e}_1|^{s-n}-|y-{\bf e}_1|^{s-n}\right|}{|x-y|^{n+sp}}\,dy\,dx=C(n,s,p)<\infty.
$$
 Moreover, we recall that it has been proven in \cite[p.\,73]{Silvestre2007CPAM} that if $\phi\in C_c^\infty(\rn)$ then
$$
|\big((-\Delta)^\frac s2 \phi\big)(z)|\le C(1+|z|)^{-n-s}
$$
holds for all $z\in\rn$ and for some constant $C\in(0,\infty)$. From this last fact, it follows that
\begin{align}\label{eq-Fu02}
  \int_\rn |z|^{s-n} |\big((-\Delta)^\frac s2 \phi\big)(z)|\,dz
  &\ls \int_\rn |z|^{s-n}(1+|z|)^{-n-s}\,dz\\
  &\ls \int_{|z|<1} |z|^{s-n}\,dz+\int_{|z|\ge 1} |z|^{-2n}\,dz\notag\\
  &\ls 1.\notag
\end{align}
Combining \eqref{eq-Fu01} and \eqref{eq-Fu02} shows that \eqref{eq-claim} is true.

Due to \eqref{eq-claim}, we can apply the Fubini theorem  to exchange the order of integrals in \eqref{eq-Fu0}, which gives
 \begin{align}\label{eq-Fu00}
    &\int_{\rn}\int_{\rn}\frac{|u_0(x)-u_0(y)|^{p-2}(u_0(x)-u_0(y))(\phi(x)-\phi(y))}{|x-y|^{n+sp}}\,dy\,dx\\
     &\quad=c_{n,s}\int_{\rn}\left(\int_{\rn}\int_{\rn}\frac{F(u_0)(x,y)\left(|x-z|^{s-n}-|y-z|^{s-n}\right)}{|x-y|^{n+sp}} \,dy\,dx\right)\big((-\Delta)^\frac s2 \phi\big)(z)\,dz. \notag
  \end{align}
For any $z\neq 0$, we proceed the same argument as in \eqref{eq-Fu01} and obtain
\begin{align*}
 & \int_{\rn}\int_{\rn}\frac{F(u_0)(x,y)\left(|x-z|^{s-n}-|y-z|^{s-n}\right)}{|x-y|^{n+sp}} \,dy\,dx\\
  &\quad =|z|^{s-n} \int_{\rn}\int_{\rn}\frac{F(u_0)(x,y)\left(|x-{\bf e}_1|^{s-n}-|y-{\bf e}_1|^{s-n}\right)}{|x-y|^{n+sp}} \,dy\,dx. \notag
\end{align*}
Upon recalling that
$$c_{n,s,p}^\ast= \frac 12\int_{\rn}\int_{\rn}\frac{F(u_0)(x,y)\left(|x-{\bf e}_1|^{s-n}-|y-{\bf e}_1|^{s-n}\right)}{|x-y|^{n+sp}} \,dy\,dx$$
and applying \eqref{eq-phi},
we now continue \eqref{eq-Fu00} with
\begin{align*}
     &\frac 12\int_{\rn}\int_{\rn}\frac{|u_0(x)-u_0(y)|^{p-2}(u_0(x)-u_0(y))(\phi(x)-\phi(y))}{|x-y|^{n+sp}}\,dy\,dx\\
    &\quad=c_{n,s,p}^\ast c_{n,s}\int_{\rn}|z|^{s-n}\big((-\Delta)^\frac s2 \phi\big)(z)\,dz\\
    &\quad = c_{n,s,p}^\ast {\mathcal I}_s\big((-\Delta)^\frac s2 \phi\big)(0)\\
    &\quad =c_{n,s,p}^\ast\phi(0).
  \end{align*}
  This ends the  proof of Theorem \ref{thm-solution}.
\end{proof}

\section{Nonexistence of nontrivial nonnegative solutions to the fractional $p$-Laplace inequality}\label{sec3}

The main goal of this section is to prove Theorem \ref{thm-Liouville}(i)-(ii).
Let us be more precise. In Sections \ref{ss3.1} and  \ref{ss3.2} below, we respectively recall
the weak Harnack inequality for  weak supersolutions, and also comparison theorems for both weak supersolutions and for super/sub-harmonic functions that are from
\cite{KorvenpaaKuusiPalatucci2017MathAnn, KuusiMingioneSire2015CMP}. Additionally, we show that the fundamental solution in Theorem \ref{thm-solution} is a harmonic function (see Proposition \ref{prop-harmonic} below). Consequently,
in  Section \ref{ss3.3} below, we not only establish a basic integral estimate involving  the weak supersolutions (see Theorem \ref{thm-harnack} below),
but also compare the weak supersolution with the fundamental solution via the comparison theorems, demonstrating that any weak supersolution is bounded from below by the fundamental solution at infinity (see Theorem \ref{thm-mr} below).
Using these two technical results (Theorems \ref{thm-harnack} and \ref{thm-mr}), we prove parts (i) and (ii) of Theorem \ref{thm-Liouville} in Section \ref{ss3.4} below.


\subsection{Weak supersolutions}\label{ss3.1}

\begin{definition}\label{def-weaksolu}
  Let $s\in (0,1)$, $p\in(1,\infty)$ and $\Omega$ be an open set in $\rn$.
  \begin{enumerate}
    \item[\rm (i)] A function $u\in W_\loc^{s,p}(\Omega)$ such that $u_-\in L_{sp}^{p-1}(\rn)$
    is called a
  {\it weak $(s,p)$-supersolution} in $\Omega$ if
    \begin{align*}
  (-\Delta)^s_pu\ge 0 \quad\text{weakly in}\ \ \Omega
  \end{align*}
that is, for all nonnegative functions $\phi\in C_c^\infty(\Omega)$,
\begin{align}\label{eq-spuphi}
  \langle (-\Delta)^s_pu,\ \phi\rangle :=  \frac 12 \int_{\rn}\int_{\rn}\frac{|u(x)-u(y)|^{p-2}(u(x)-u(y))(\phi(x)-\phi(y))}{|x-y|^{n+sp}}\,dy\,dx\ge 0.
\end{align}

  \item[\rm (ii)] Similarly,
   $u\in W_\loc^{s,p}(\Omega)$ such that $u_+\in L_{sp}^{p-1}(\rn)$
    is called a
  {\it weak $(s,p)$-subsolution}  in $\Omega$
   if
   \begin{align*}
  (-\Delta)^s_pu\le 0 \quad\text{weakly in}\ \ \Omega,
  \end{align*}
that is, for all $0\le\phi\in C_c^\infty(\Omega)$,
   $$\langle (-\Delta)^s_pu,\ \phi\rangle\le 0.$$

   \item[\rm (iii)]  If $u\in W_\loc^{s,p}(\Omega)\cap L_{sp}^{p-1}(\rn)$ is said to be a {\it weak $(s,p)$-solution} in $\Omega$
       if it is both a weak $(s,p)$-supersolution and a weak $(s,p)$-subsolution in $\Omega$. In this case,  $u$ satisfies
  \begin{align*}
  (-\Delta)^s_pu= 0 \quad\text{weakly in}\ \ \Omega,
  \end{align*}
  that is, for all $0\le\phi\in C_c^\infty(\Omega)$,
   $$\langle (-\Delta)^s_pu,\ \phi\rangle= 0.$$
  \end{enumerate}
\end{definition}

\begin{remark}\label{rem-Deltaspu}
For a sufficiently nice function $u$ (for instance, $u\in C_c^\infty(\rn)$),  it follows from Proposition \ref{prop-ChenLi} that $(-\Delta)^s_pu$ is pointwisely well-defined and, for any $0\le\phi\in C_c^\infty(\rn)$,
\begin{align*}
\int_\rn (-\Delta)^s_pu(x)\phi(x)\,dx
&= \mathrm{p.\,v.}\int_{\rn}\int_{\rn}\frac{|u(x)-u(y)|^{p-2}(u(x)-u(y))\phi(x)}{|x-y|^{n+sp}}\,dy\,dx.
\end{align*}
Interchanging the roles of $x$ and $y$, we  also have
\begin{align*}
\int_\rn (-\Delta)^s_pu(x)\phi(x)\,dx
&= \mathrm{p.\,v.}\int_{\rn}\int_{\rn}\frac{|u(y)-u(x)|^{p-2}(u(y)-u(x))\phi(y)}{|y-x|^{n+sp}}\,dx\,dy.
\end{align*}
Adding these two expressions and dividing both sides by $2$ yields the symmetric weak formulation
\begin{align*}
\int_\rn (-\Delta)^s_pu(x)\phi(x)\,dx
=\frac 12\int_{\rn}\int_{\rn}\frac{|u(x)-u(y)|^{p-2}(u(x)-u(y))(\phi(x)-\phi(y))}{|x-y|^{n+sp}}\,dy\,dx.
\end{align*}
This justifies the definition of the bracket in \eqref{eq-spuphi}.
\end{remark}

We will need the following version of nonlocal weak Harnack inequality from
 \cite[Theorem~6]{KorvenpaaKuusiPalatucci2017MathAnn}
 (see also \cite[Theorem~1.2]{DiCastroKuusiPalatucci2014JFA}).

\begin{lemma}\label{lem-Harnack}
Let $s\in (0,1)$, $p\in (1,\infty)$  and $\Omega$ be an open set in $\rn$. Assume that
 $0\le u\in W_\loc^{s,p}(\Omega)$  such that
		\begin{equation*}
		(-\Delta)^s_pu\ge 0\ \ \text{weakly in}\ \ \Omega.
		\end{equation*}
Then,  for any ball $B$ satisfying $2B\Subset \Omega$ and for any $t\in (0, \infty)$ satisfying
\begin{align}\label{eq-bar-t}
t<\bar t:= \begin{cases}
             \frac{n(p-1)}{n-sp} & \text{as}\ sp<n;\\
             +\infty  & \text{as}\ sp\ge n,
           \end{cases}
\end{align}
it holds that
\begin{equation}\label{eq-WH}
\left(\fint_{B} u(x)^t\, dx\right)^\frac 1t
\le C \mathop{\mathrm{essinf}}_{2B}u,
\end{equation}
where $C$ is a positive constant  depending only on $s,p$ and $n$.
\end{lemma}

Next, we recall the following lemma from \cite[Lemma~7.3]{KuusiMingioneSire2015CMP} (see also
\cite[Lemma~11]{KorvenpaaKuusiPalatucci2017MathAnn}).

\begin{lemma}\label{lem-hanack2}
Let $s\in (0,1)$, $p\in (1,\infty)$  and $\Omega$ be an open set in $\rn$. Assume that
 $0\le u\in W_\loc^{s,p}(\Omega)$  such that
\begin{equation*}
		(-\Delta)^s_pu\ge 0\ \ \text{weakly in}\ \ \Omega.
		\end{equation*}
Suppose that $h\in (0, s)$ and $\nu\in (0, \bar \nu)$, where
\begin{align}\label{eq-bar-nu}
\bar \nu =\min\left\{\frac{n(p-1)}{n-s},\ p\right\}= \begin{cases}
             \frac{n(p-1)}{n-s} & \text{as}\ sp<n;\\
             p  & \text{as}\ sp\ge n.
           \end{cases}
\end{align}
Then, there exists a positive constant $C$, depending on $n,s,p, s-h$ and $\bar\nu-\nu$, such that for any ball $B$ of radius $r\in(0,\infty)$ and $4B\Subset\Omega$,
\begin{align}\label{eq-Holder}
\left(\fint_{2B}\int_{2B}\frac{|u(x)-u(y)|^\nu}{|x-y|^{n+h\nu}}\,dy\,dx\right)^{1/\nu}\le C r^{-h}\einf_{B} u.
\end{align}
Let us emphasize that the constant $C$ in \eqref{eq-Holder} is independent of $u$, $B$ and $\Omega$.
\end{lemma}

\subsection{Comparison principles}\label{ss3.2}

Let us begin with the following comparison principle  for weak solutions, which is  proved
in
\cite[Lemma~6]{KorvenpaaKuusiPalatucci2017MathAnn} (see also \cite[Lemma~9]{LindgrenLindqvist2014CVPDE}).

\begin{lemma}\label{lem-KKP-CP2}
 Let $s\in(0,1)$ and $p\in(1,\infty)$. Suppose that $D$  and $\Omega$ are open subsets of $\rn$ such that $D\Subset \Omega$.
For any $u, v\in W_{\loc}^{s,p}(\Omega)$ such that $u_-, v_-\in L_{sp}^{p-1}(\rn)$, if
 $u\ge v$ almost everywhere on $\rn\setminus D$
  and
   $$\langle (-\Delta)^s_pu,\ \phi\rangle\ge \langle (-\Delta)^s_pv,\ \phi\rangle\quad\text{for all}\ \ 0\le\phi\in C_c^\infty(\Omega),$$
 then $u\ge v$ almost everywhere on $D$ (and, hence, on $\rn$).
\end{lemma}

The problem is that the fundamental solution  in \eqref{eq-U}, which behaves like $|x|^{\frac{sp-n}{p-1}}$,  may fail to belong to $L^p(\Omega)$ and $W^{s,p}(\Omega)$. In order to  have  a comparison theorem applicable to the fundamental solution, we recall the notion of
$(s, p)$-harmonic functions (see \cite{KorvenpaaKuusiPalatucci2017MathAnn, KorvenpaaKuusiLindgren2019JMPA}).

\begin{definition}\label{def-harmonic}
 A function $u :\ \rn\to [-\infty,+\infty]$ is called an {\it $(s, p)$-superharmonic
function} in an open set $\Omega\subseteq\rn$ if it satisfies the following conditions:
\begin{enumerate}
\item[\rm (i)] $u < +\infty$ a.e. on $\rn$ and $u > -\infty$ everywhere in $\Omega$;
\item[\rm (ii)] $u$ is lower semicontinuous in $\Omega$;
\item[\rm (iii)] $u$ satisfies the comparison principle in $\Omega$ against solutions, that is,
if $D$ is an open set of $\rn$, $D\Subset\Omega$ and $v \in C(\overline D)$ is a weak $(s,p)$-solution in $D$  and $u\ge  v$ on $\partial D$ and almost everywhere on $\rn\setminus D$, then
$u\ge  v$ in $D$;
\item[\rm (iv)]  $u_-=-\min\{u,0\}$ belongs to $L^{p-1}_{sp}(\rn)$.
\end{enumerate}
A function $u$ is said to be {\it $(s, p)$-subharmonic} in $\Omega$ if $-u$ is $(s, p)$-superharmonic in $\Omega$. If both $u$ and $-u$ are $(s, p)$-superharmonic  in $\Omega$, then $u$ is said to be {\it $(s, p)$-harmonic}  in $\Omega$.
\end{definition}

Note that constant functions are always $(s,p)$-harmonic  in $\rn$.
As a consequence of Lemma \ref{lem-KKP-CP2}, we can show that fundamental solutions given  in \eqref{eq-U}
are
$(s,p)$-harmonic  in $\rn\setminus\{0\}$.

\begin{proposition}\label{prop-harmonic}
 Let $s\in (0,1)$ and $p\in (1,\infty)$ such that $sp\le n$. Then the function $u_0$ in \eqref{eq-u0} (and, hence, the fundamental solution $u$  in \eqref{eq-U})
 is $(s,p)$-harmonic  in $\rn\setminus\{0\}$.
\end{proposition}

\begin{proof}
Let $u_0$ be as in \eqref{eq-u0}.
It suffices to validate that  $u_0$ satisfies Definition \ref{def-harmonic}(iii); the other items of Definition \ref{def-harmonic} are obvious for $u_0$.

Let $D$ be an open set of $\rn$ such that $D\Subset \rn\setminus\{0\}$, which implies that $\overline D$ is compact and $0\notin \overline D$.
From this and the continuity of $u_0$  on $\rn\setminus\{0\}$,
it follows that
$u_0\in L^{p}(D)$.  Further, applying \eqref{eq-u0xy} and Lemma \ref{lem1} gives
\begin{align*}
[u_0]_{W^{s,p}(D)}^p
&=\int_D\int_D \frac{|u_0(x)-u_0(y)|^p}{|x-y|^{n+sp}}\,dy\,dx\\
&\ls \int_D \int_{|y-x|<1} \frac{|x|^{(\frac{sp-n}{p-1}-1)p} |x-y|^p}{|x-y|^{n+sp}}\,dy\,dx\\
&\quad
+\int_D\int_{|y-x|\ge 1}\frac{u_0(x)^p+u_0(y)^p}{|x-y|^{n+sp}}\,dy\,dx\\
&\ls \int_D |x|^{(\frac{sp-n}{p-1}-1)p}\,dx
+ \int_D u_0(x)^p\,dx+\int_D u_0(y)^p\,dy\\
&<\infty.
\end{align*}
This proves $u_0\in W^{s,p}(D)$.
Also, it is easy to see that $u_0\in L_{sp}^{p-1}(\rn)$.
According to
  Theorem \ref{thm-solution}, for any $0\le \phi\in C_c^\infty(D)$, it holds that
  $$\langle (-\Delta)^s_pu_0,\ \phi\rangle=c_{n,s,p}^\ast\phi(0)=0,$$
     that is,
\begin{equation*}
		(-\Delta)^s_pu_0=0\quad \text{weakly in}\ \, D.
\end{equation*}
By these facts, we derive from Definition \ref{def-weaksolu} that $u_0$ is a weak $(s,p)$-solution    in $D$.

Suppose that $v \in C(\overline D)$ is a weak $(s,p)$-solution   in $D$ such that
$u_0\ge  v$ on $\partial D$ and almost everywhere on $\rn\setminus D$.
Since $v$ is a  weak $(s,p)$-solution   in $D$,
it follows from  Definition \ref{def-weaksolu} that $v\in W_\loc^{s,p}(D)$ and $v_-\in L_{sp}^{p-1}(\rn)$.
For any $\epsilon>0$, by the continuity of $u_0$ and $v$ on the compact set $\overline D$, there exists an open set $D_\epsilon\Subset D$ such that
$$
u_0\ge v-\epsilon \quad \text{pointwisely on}\ D\setminus D_\epsilon.
$$
This allows us to apply Lemma \ref{lem-KKP-CP2} to compare $u_0$ and $v-\epsilon$, which implies
$$
u_0\ge v-\epsilon \quad \text{almost everywhere on}\ \rn.
$$
Letting $\epsilon\to0$ gives that $
u_0\ge v$ almost everywhere on $\rn$. This proves that $u_0$ satisfies Definition \ref{def-harmonic}(iii).  Thus,
$u_0$
 is $(s,p)$-superharmonic  in $\rn\setminus\{0\}$.

 In a similar manner, we can show that  $u_0$
 is $(s,p)$-lowerharmonic  in $\rn\setminus\{0\}$. Altogether, we obtain that
 $u_0$
 is $(s,p)$-harmonic  in $\rn\setminus\{0\}$.
\end{proof}

In addition to Lemma \ref{lem-KKP-CP2}, we will also need the following general comparison principle for super/sub-harmonic functions (see \cite[Theorem~16]{KorvenpaaKuusiPalatucci2017MathAnn}).

\begin{lemma}\label{lem-KKP-CP-general}
Let $s\in(0,1)$, $p\in(1,\infty)$ and $\Omega$ be an  open subset of $\rn$.
Assume that $u$ is $(s,p)$-superharmonic in $\Omega$ and $v$ is $(s,p)$-subharmonic in $\Omega$.
If $u\ge v$ a.e. on  $\rn\setminus \Omega$ and
\begin{align}\label{eq1-CP}
\liminf_{\gfz{y\in\Omega}{y\to x}} u(y)\ge \limsup_{\gfz{y\in\Omega}{y\to x}} v(y) \quad \text{for all}\  x\in\partial\Omega
\end{align}
such that both sides are not simultaneously $+\infty$ or $-\infty$,  then $u\ge v$ a.e. on  $\rn$.
\end{lemma}

At the end of this section, we state the following results that are from \cite[Theorems~1, 9, 12, 13]{KorvenpaaKuusiPalatucci2017MathAnn}, which
clarify the relations between $(s,p)$-harmonic functions in Definition \ref{def-harmonic} and weak solutions in Definition \ref{def-weaksolu}.

\begin{lemma}\label{lem-SHSL}
Let $s\in(0,1)$, $p\in(1,\infty)$ and $\Omega$ be an  open subset of $\rn$.
\begin{enumerate}
  \item[\rm (i)]  If $u$ is an $(s,p)$-superharmonic function in $\Omega$ such that either $u$ is locally bounded or $u\in W^{s,p}_\loc(\Omega)$, then $u$ is a weak $(s,p)$-supersolution in $\Omega$, that is,
      \begin{equation*}
		(-\Delta)^s_pu\ge 0\ \ \text{weakly in}\ \ \Omega.
		\end{equation*}

      \item[\rm (ii)] If $u\in W^{s,p}(\Omega)$ is a weak $(s,p)$-supersolution in $\Omega$,
then
 \begin{align*}
 u(x)=\liminf_{\gfz{y\in\Omega}{y\to x}} u(y) \quad \text{for a.e.}\  x\in\Omega.
\end{align*}
Moreover,  the lower semicontinuous representative of $u$, denote by
$$\tilde u(x)=\liminf_{\gfz{y\in\Omega}{y\to x}} u(y)  \quad \text{for all}\  x\in\Omega,$$
is an $(s,p)$-superharmonic function in $\Omega$. In particular, for any open set $D\Subset\Omega$,
$$\inf_D \tilde u=\einf_D \tilde u.$$

  \item[\rm (iii)] A function $u$ is $(s,p)$-harmonic  in $\Omega$ if and only if $u$ is a continuous weak solution  in $\Omega$.

\end{enumerate}

\end{lemma}

\begin{remark}
In the proof of Theorem \ref{thm-mr} below, we will apply Lemma \ref{lem-KKP-CP-general} to compare $\tilde u$ and the fundamental solution in certain domain $\Omega$, with $\tilde u$ being the lower semicontinuous representative of a weak $(s,p)$-supersolution $u$ in $\Omega$.
\end{remark}

\subsection{Behaviour of weak supersolutions}\label{ss3.3}

As a consequence of the previous Lemmas \ref{lem-Harnack} and \ref{lem-hanack2}, we show the following theorem for weak supersolutions, which will be used to show the nonexistence of nontrivial nonnegative solutions to the fractional $p$-Laplace inequality \eqref{eq-p-Lapineq} under the case $sp\ge n$ or $sp<n$ and $q<\frac{n(p-1)}{n-sp}$.

\begin{theorem}\label{thm-harnack}
Let $s\in (0,1)$, $p\in (1,\infty)$  and $\Omega$ be an open set in $\rn$. Assume that
 $0\le u\in W_\loc^{s,p}(\Omega)$  such that
\begin{equation*}
		(-\Delta)^s_pu\ge 0\ \ \text{weakly in}\ \ \Omega.
		\end{equation*}
  Suppose that $B=B(x_0, r)$ such that $8B\Subset\Omega$, where $x_0\in\rn$ and $r\in (0,\infty)$.
For any $x\in\rn$, define
$$
\phi_r(x):=\eta\left(\frac{x-x_0}{ r}\right),
$$
where $\eta\in C_c^\infty(\rn)$ such that
$\supp\eta\subseteq B(0,2)$, $\eta\equiv 1$ on $B(0,1)$ and $0\le\eta\le 1$.
Then
\begin{align}\label{eq-mr110}
&\int_\rn\int_\rn \frac{|u(x)-u(y)|^{p-1}|\phi_r(x)-\phi_r(y)|}{|x-y|^{n+sp}}\,dy\,dx \le C r^{n-sp}\left(\einf_{2B} u\right)^{p-1},
\end{align}
where $C$ is  a positive constant depending only on $s,p$, $n$  and  $\eta$.
\end{theorem}

\begin{proof}

For the double-integral in the  left hand side of \eqref{eq-mr110}, we split it as $$
\left(\int_{4B}\int_{4B}
+\int_{4B}\int_{(4B)^c}
+ \int_{(4B)^c}\int_{4B}
+\int_{(4B)^c}\int_{(4B)^c}
 \right) \cdots.
$$
Denote these four double-integrals by ${\rm I}_1$, ${\rm I}_2$, ${\rm I}_3$
and ${\rm I}_4$, respectively.
Observe that  ${\rm I}_2={\rm I}_3$ by symmetry.
Moreover, since $\supp\phi_r\subseteq 2B$, it follows that
$$
{\rm I}_4=
\int_{(4B)^c}\int_{(4B)^c}
 \frac{|u(x)-u(y)|^{p-1}|\phi_r(x)-\phi_r(y)|}{|x-y|^{n+sp}}\,dy\,dx
=0.
$$
Thus, to obtain \eqref{eq-mr110}, it suffices to estimate ${\rm I}_1$ and ${\rm I}_2$.

In order to estimate ${\rm I}_1$, we will apply  Lemma \ref{lem-hanack2}.
Observing that $\bar \nu$ defined in \eqref{eq-bar-nu} is strictly larger than $p-1$, we take $\bar p$ such that
$$
p-1<\bar p< \bar \nu.
$$
Since $s\in (0,1)$, we can take $h\in (0,s)$ sufficiently close to $s$ to satisfy
$$
p(s-h)+h<1.
$$
Applying the H\"older inequality with $\frac{1}{\bar p/(p-1)}+ \frac 1{\bar p/(\bar p-p+1)}=1$, we obtain
\begin{align*}
  {\rm I}_1 &= \int_{4B}\int_{4B}
 |u(x)-u(y)|^{p-1} |\phi_r(x)-\phi_r(y)||x-y|^{h\bar p-sp} \,
 \frac{dy\,dx}{|x-y|^{n+h \bar p}}\\
 &\le \left(
 \int_{4B}\int_{4B}
 \frac{|u(x)-u(y)|^{\bar p}}{|x-y|^{n+h \bar p}}\,dy\,dx
 \right)^{\frac{p-1}{\bar p}}\\
 &\quad\times
 \left(
 \int_{4B}\int_{4B}
 \frac{(|\phi_r(x)-\phi_r(y)||x-y|^{h\bar p-sp})^{\frac{\bar p}{\bar p-p+1}}}{|x-y|^{n+h \bar p}}\,dy\,dx
 \right)^{\frac{\bar p-p+1}{\bar p}}
  \\
  &\ls \left( r^{\frac{n}{\bar p}-h}\einf_{2B} u\right)^{p-1}
  \left(
 \int_{4B}\int_{4B}|\phi_r(x)-\phi_r(y)|^{\frac{\bar p}{\bar p-p+1}}
 |x-y|^{- \frac{\bar p[p(s-h)+h]}{\bar p-p+1}-n}\,dy\,dx
 \right)^{\frac{\bar p-p+1}{\bar p}},
\end{align*}
where the last step is due to \eqref{eq-Holder}. By the mean value theorem and the definition of $\phi_r$,
one has
\begin{align*}
  |\phi_r(x)-\phi_r(y)|\le \sup_{\theta\in(0,1)} |\nabla \phi_r(x+\theta(y-x))| |x-y|
  \le \|\nabla\phi\|_{L^\infty(\rn)} r^{-1}|x-y|
\end{align*}
and, hence, by $p(s-h)+h<1$,
\begin{align*}
&  \left(
 \int_{4B}\int_{4B}|\phi_r(x)-\phi_r(y)|^{\frac{\bar p}{\bar p-p+1}}
 |x-y|^{- \frac{\bar p[p(s-h)+h]}{\bar p-p+1}-n}\,dy\,dx
 \right)^{\frac{\bar p-p+1}{\bar p}}\\
  &\quad \ls  r^{-1}
 \left(\int_{4B}\int_{|x-y|< 8 r}
 |x-y|^{(1-p(s-h)-h)\frac{\bar p}{\bar p-p+1}-n}\,dy\,dx\right)^{\frac{\bar p-p+1}{\bar p}}\\
 &\quad \ls  r^{-1}
 \left(\int_{4B}\int_0^{8 r}
\rho^{(1-p(s-h)-h)\frac{\bar p}{\bar p-p+1}-1}\,d\rho\,dx\right)^{\frac{\bar p-p+1}{\bar p}}\\
 &\quad \ls  r^{-1}
 \left(\int_{4B}r^{(1-p(s-h)-h)\frac{\bar p}{\bar p-p+1}}\,dx\right)^{\frac{\bar p-p+1}{\bar p}}\\
 &\quad \simeq  r^{-p(s-h)-h+n\frac{\bar p-p+1}{\bar p}}.
\end{align*}
Thus, we have
\begin{align}\label{eq-I1}
  {\rm I}_1
  \ls \left( r^{\frac{n}{\bar p}-h}\einf_{2B} u\right)^{p-1}
  r^{-p(s-h)-h+n\frac{\bar p-p+1}{\bar p}}
  \simeq
 r^{n-sp} \left( \einf_{2B} u\right)^{p-1}.
  \end{align}

   Now, we turn to the estimate of ${\rm I}_2$. Invoking $\supp \phi_r\subseteq 2B$ and $0\le\phi_r\le 1$, we write
   \begin{align*}
   {\rm I}_2
   &=\int_{4B}\int_{(4B)^c} \frac{|u(x)-u(y)|^{p-1}|\phi_r(x)|}{|x-y|^{n+sp}}\,dy\,dx
   \le\int_{2B}\int_{(4B)^c} \frac{|u(x)-u(y)|^{p-1}}{|x-y|^{n+sp}}\,dy\,dx.
   \end{align*}
     Since $u\ge 0$, it follows that either $|u(x)-u(y)|\le u(x)$
   or $|u(x)-u(y)|\le u(y)$, which further gives
    \begin{align*}
   {\rm I}_2
   &\le \int_{2B}\int_{(4B)^c} \frac{u(x)^{p-1}}{|x-y|^{n+sp}}\,dy\,dx
   +\int_{2B}\int_{(4B)^c} \frac{u(y)^{p-1}}{|x-y|^{n+sp}}\,dy\,dx.
   \end{align*}
 For any $x\in 2B$ and $y\notin 4B$,  we observe that $|y-x|>2r$. This, along with \eqref{eq-basic2} and the nonlocal weak  Harnack inequality \eqref{eq-WH} (taking $t$ therein to be $p-1$), further derives
\begin{align*}
\int_{2B}\int_{(4B)^c} \frac{u(x)^{p-1}}{|x-y|^{n+sp}}\,dy\,dx
&\le \int_{2B}\int_{|y-x|>2r} \frac{u(x)^{p-1}}{|x-y|^{n+sp}}\,dy\,dx\\
 & \ls r^{-sp} \int_{2B}  u(x)^{p-1}\,dx\\
&\ls r^{n-sp} \left(\einf_{4B} u\right)^{p-1}\\
&\ls r^{n-sp} \left(\einf_{2B} u\right)^{p-1}.
\end{align*}
Meanwhile, if $x\in 2B$ and $y\notin 4B$, then
 $|y-x_0|\le |x-y| +|x-x_0|\le |x-y|+|y-x_0|/2$ and, hence,
  $|y-x_0|/2\le |x-y|$.
So, upon splitting $(4B)^c$ into annulus, we then utilize  \eqref{eq-WH}  again to derive
\begin{align*}
 \int_{2B}\int_{(4B)^c} \frac{u(y)^{p-1}}{|x-y|^{n+sp}}\,dy\,dx
 &\le 2^{n+sp}\int_{2B}\int_{(4B)^c} \frac{u(y)^{p-1}}{|y-x_0|^{n+sp}}\,dy\,dx\\
 &=2^{n+sp}\omega_{n-1} r^n \sum_{j=1}^\infty \int_{4^j r\le |y-x_0|< 4^{j+1}r} \frac{u(y)^{p-1}}{|y-x_0|^{n+sp}}\,dy\\
 &\le 2^{n+sp}\omega_{n-1} r^n\sum_{j=1}^\infty (4^jr)^{-sp-n}
 \int_{|y-x_0|< 4^{j+1}r} u(y)^{p-1}\,dy\\
 &\ls  r^n\sum_{j=1}^\infty (4^jr)^{-sp}
 \left(\einf_{{\mathbb B}_{4^{j+2}r}} u\right)^{p-1}\\
 &\ls  r^{n-sp} \left(\einf_{2B} u\right)^{p-1}.
\end{align*}
Combining the last two formulae gives
\begin{align}\label{eq-I2}
{\rm I}_2 \ls  r^{n-sp} \left(\einf_{2B} u\right)^{p-1}.
\end{align}

By \eqref{eq-I1} and \eqref{eq-I2}, together with ${\rm I}_2={\rm I}_3$ and
${\rm I}_4=0$, we then deduce that \eqref{eq-mr110} holds.
\end{proof}

The following result concerns the existence and continuity of solutions to a Dirichlet problem,
whose proof is based on some already known results of the obstacle problem for
$(-\Delta)^s_p$ that has been obtained in \cite{KorvenpaaKuusiPalatucci2016CVPDE, DiCastroKuusiPalatucci2016AIHP}.

\begin{proposition}\label{prop-DP}
Let $s\in(0,1)$, $p\in(1,\infty)$ and $\Omega$ be a bounded open set in $\rn$.
Given any function $g\in W^{s,p}(\rn)$, there exists a unique solution $u\in W^{s,p}(\rn)$ to the following
Dirichlet problem
\begin{align}\label{eq-DP0}
\begin{cases}
(-\Delta)^s_p u=0 &\quad\text{weakly in}\ \Omega;\\
u=g  &\quad\text{a.e on}\ \, \rn\setminus \Omega.
\end{cases}
\end{align}
Further, if we assume that $g$ is continuous on $\rn$ and the set $\rn\setminus\Omega$ satisfies the {\it measure density condition}, that is, there exists $\delta_\Omega\in(0,1)$ and $r_\Omega\in(0,\infty)$ such that for all $x\in\partial\Omega$,
\begin{align}\label{eq-mdc}
\inf_{r\in(0, \, r_\Omega)}\frac{|(\rn\setminus\Omega)\cap B(x,r)|}{|B(x,r)|}\ge \delta_\Omega,
\end{align}
then the solution $u$  to \eqref{eq-DP0} is almost everywhere equal to a continuous function on $\rn$.
\end{proposition}

\begin{proof}
For any $g\in W^{s,p}(\rn)$, define
$$
\mathcal K_g(\Omega):=\left\{w\in W^{s,p}(\rn):\ \, w=g\ \text{a.e. on}\ \rn\setminus\Omega\right\}.
$$
According to \cite[Theorem~2.3]{DiCastroKuusiPalatucci2016AIHP}, there always exists a unique element $u\in \mathcal K_g(\Omega)$ such that
\begin{align}\label{eq-minimizer}
[u]_{ W^{s,p}(\rn)}=\min\left\{[v]_{ W^{s,p}(\rn)}:\, v\in \mathcal K_g(\Omega)\right\},
\end{align}
and, moreover, this minimizer  $u$ uniquely solves the Dirichlet problem \eqref{eq-DP0}.
If $g$ is continuous in $\Omega$, then it is proved in \cite[Theorem~1.2]{DiCastroKuusiPalatucci2016AIHP} that $u$ is also continuous in $\Omega$. Due to  the continuity of $g$  on $\rn$ and the second condition of \eqref{eq-DP0}, one has the continuity of $u$ on the interior of $\rn\setminus\Omega$. The problem is to obtain the continuity of $u$ on the boundary $\partial\Omega$. This issue has been addressed in \cite{KorvenpaaKuusiPalatucci2016CVPDE}.

Let $\Omega$ and $\Omega'$ be two open bounded sets in $\rn$ such that $\Omega\Subset\Omega'$.
With a prescribed function $g\in W^{s,p}(\Omega')\cap L_{sp}^{p-1}(\rn)$, define
$$
\mathcal K_g(\Omega,\Omega'):=\left\{w\in W^{s,p}(\Omega'):\ \, w=g\ \text{a.e. on}\ \rn\setminus\Omega\right\}.
$$
As was proved in \cite[Theorem~1]{KorvenpaaKuusiPalatucci2016CVPDE},
there is a unique solution $w\in \mathcal K_g(\Omega,\Omega') $ to the obstacle problem that satisfies
\begin{align}\label{eq-obsta}
\langle (-\Delta)^s_p w,\, v-w\rangle\ge 0\quad \text{for all}\ v\in \mathcal K_g(\Omega,\Omega'),
\end{align}
where the bracket $\langle \cdot,\cdot\rangle$ is understood as in \eqref{eq-spuphi}, that is,
\begin{align*}
&\langle (-\Delta)^s_p w,\, v-w\rangle\\
&\quad =
\frac 12 \int_{\rn}\int_{\rn}\frac{|w(x)-w(y)|^{p-2}(w(x)-w(y))[(v-w)(x)-(v-w)(y)]}{|x-y|^{n+sp}}
\,dy\,dx.
\end{align*}
For this solution
$w$ to the obstacle problem \eqref{eq-obsta} in $\mathcal K_g(\Omega,\Omega')$, if we assume that $g$ is continuous on $\Omega'$ and $\rn\setminus\Omega$ satisfies the measure density condition \eqref{eq-mdc},
then it is proved in \cite[Theorem~9]{KorvenpaaKuusiPalatucci2016CVPDE} that
$w$  is continuous on $\Omega'$.

 Observe that
$\mathcal K_g(\Omega)\subseteq \mathcal K_g(\Omega,\Omega')$ by their definitions. Now, since we have a stronger assumption of $g\in W^{s,p}(\rn)$, we claim that
$$w\in \mathcal K_g(\Omega).$$
To show this claim, it suffices to verify that $w\in W^{s,p}(\rn)$.
Indeed, by using
$\Omega\Subset\Omega'$ and $w\in \mathcal K_g(\Omega,\Omega')$, we are able to show that
 \begin{align}\label{eq-w-glob}
\|w\|_{W^{s,p}(\rn)}\ls (1+d^{-1}) \left(\|w\|_{W^{s,p}(\Omega')}+\|g\|_{W^{s,p}(\rn)}\right),
\end{align}
where $d:=\text{dist}(\Omega,\rn\setminus\Omega')$.   The details are as follows.
On the one hand, from $w=g$ a.e. on $\rn\setminus \Omega$
and $g\in W^{s,p}(\rn)\subseteq L^p(\rn)$,
together with
$$w\in \mathcal K_g(\Omega,\Omega')\subseteq W^{s,p}(\Omega')\subseteq L^p(\Omega'),$$
it follows that
$w\in L^p(\rn)$ with
$$
\|w\|_{L^{p}(\rn)}\le \|w\|_{L^{p}(\Omega')}+\|g\|_{L^{p}(\rn)}.
$$
On the other hand, via using $w=g$ a.e. on $\rn\setminus \Omega$, we write
\begin{align*}
[w]_{W^{s,p}(\rn)}^p
&=\int_{\Omega'}\int_{\Omega'} \frac{|w(x)-w(y)|^p}{|x-y|^{n+sp}}\,dy\,dx
+ \int_{\rn\setminus \Omega'}\int_{\rn\setminus\Omega'} \frac{|g(x)-g(y)|^p}{|x-y|^{n+sp}}\,dy\,dx\\
&\quad
+2\int_{\Omega'}\int_{\rn\setminus\Omega'} \frac{|w(x)-w(y)|^p}{|x-y|^{n+sp}}\,dy\,dx
\end{align*}
and
\begin{align*}
&\int_{\Omega'}\int_{\rn\setminus\Omega'} \frac{|w(x)-w(y)|^p}{|x-y|^{n+sp}}\,dy\,dx\\
&\quad=\int_{\Omega'\setminus\Omega}\int_{\rn\setminus\Omega'} \frac{|g(x)-g(y)|^p}{|x-y|^{n+sp}}\,dy\,dx
+\int_{\Omega}\int_{\rn\setminus\Omega'} \frac{|w(x)-w(y)|^p}{|x-y|^{n+sp}}\,dy\,dx,
\end{align*}
which directly implies
$$
[w]_{W^{s,p}(\rn)}^p \le
[u]_{W^{s,p}(\Omega')}^p
+3[g]_{W^{s,p}(\rn)}^p +2^{p}
 \int_{\Omega}\int_{\rn\setminus\Omega'} \frac{|w(x)|^p+|w(y)|^p}{|x-y|^{n+sp}}\,dy\,dx.
$$
Further, upon recalling that $d=\text{dist}(\Omega,\rn\setminus\Omega')$, we adopt \eqref{eq-basic2} to get
\begin{align*}
  \int_{\Omega}\int_{\rn\setminus\Omega'} \frac{|w(x)|^p}{|x-y|^{n+sp}}\,dy\,dx
  &\le \int_{\Omega} |w(x)|^p \left(\int_{|x-y|\ge d } \frac1{|x-y|^{n+sp}}\,dy \right)\,dx\\
  &
  \ls d^{-sp} \|w\|_{L^p(\Omega)}^p
\end{align*}
and
\begin{align*}
  \int_{\Omega}\int_{\rn\setminus\Omega'} \frac{|w(y)|^p}{|x-y|^{n+sp}}\,dy\,dx
  &\le \int_{\rn\setminus\Omega} |g(y)|^p \left(\int_{|x-y|\ge d } \frac1{|x-y|^{n+sp}}\,dx \right)\,dy\\
  &
  \ls d^{-sp} \|g\|_{L^p(\rn)}^p.
\end{align*}
Altogether, we obtain that $w\in W^{s,p}(\rn)$ such that \eqref{eq-w-glob} holds, whence  $w\in \mathcal K_g(\Omega)$.

As a consequence of \eqref{eq-obsta} and the H\"older inequality, we have for all $v\in \mathcal K_g(\Omega)$ that
\begin{align*}
 [w]_{ W^{s,p}(\rn)}^p
 &=\int_{\rn}\int_{\rn} \frac{|w(x)-w(y)|^{p-2}(w(x)-w(y))(w(x)-w(y))}{|x-y|^{n+sp}}\,dy\,dx\\
  &\le \int_{\rn}\int_{\rn}\frac{|w(x)-w(y)|^{p-2}(w(x)-w(y))(v(x)-v(y))}{|x-y|^{n+sp}}
\,dy\,dx\\
&\le \int_{\rn}\int_{\rn}\frac{|w(x)-w(y)|^{p-1}|v(x)-v(y)|}{|x-y|^{n+sp}}
\,dy\,dx\\
&\le \left(\int_{\rn}\int_{\rn}\frac{|w(x)-w(y)|^{(p-1)p'}}{|x-y|^{n+sp}}
\,dy\,dx\right)^{\frac 1{p'}}
\left(\int_{\rn}\int_{\rn}\frac{|v(x)-v(y)|^p}{|x-y|^{n+sp}}
\,dy\,dx\right)^{\frac 1 p}\\
&= [w]_{ W^{s,p}(\rn)}^{p-1}[v]_{ W^{s,p}(\rn)}^p.
\end{align*}
This, along with $w\in  \mathcal K_g(\Omega)$, implies
$$[w]_{ W^{s,p}(\rn)}=\min\left\{[v]_{ W^{s,p}(\rn)}:\, v\in \mathcal K_g(\Omega)\right\}.$$
Recalling the aforementioned fact that $u$ is a unique solution to
\eqref{eq-minimizer}, we are led to the conclusion that
$u=w$ a.e. on $\rn$.

Of course, in the obstacle problem \eqref{eq-obsta}, one can choose  arbitrary  bounded open sets
$\Omega'$ satisfying $\Omega\Subset\Omega'$, which in turn implies that $w$ is continuous on $\rn$. Thus, $u$ is almost everywhere equal to the continuous function $w$ on $\rn$. This ends the proof of Proposition \ref{prop-DP}.
\end{proof}

To prove the nonexistence of nontrivial nonnegative solutions to the fractional $p$-Laplace inequality \eqref{eq-p-Lapineq} under the case $sp<n$ and $q=\frac{n(p-1)}{n-sp}$, we will apply Proposition \ref{prop-DP} together with the comparison principles from Section \ref{ss3.3} to establish lower bounds for weak supersolutions at infinity.


\begin{theorem}\label{thm-mr}
  Let  $s\in (0,1)$ and $p\in (1,\infty)$ such that $sp<n$. Assume that $0\le u\in W_\loc^{s,p}(\rn)$
such that
		\begin{equation*}
		(-\Delta)^s_pu\ge 0\ \ \text{weakly in}\ \ \mathbb R^n.
		\end{equation*}
 If $u$ is  not  almost everywhere zero on $\rn$, then  the following statements hold:
\begin{enumerate}
\item[\rm (i)] there exists  $r_0\in(1,\infty)$ such that for all $r\in[r_0,\infty)$
    and for almost all $x\in\rn$,
    \begin{align*}
u(x)\ge \left(\inf_{\overline{\mathbb B}_{r}\setminus {\mathbb B}_1} u\right)  \left(1-|x|^{\frac{sp-n}{p-1}}\right);
\end{align*}
  \item[\rm (ii)]  there exists a large constant  $R_0\in(r_0,\infty)$, depending on $u$, such that for almost all $|x|\ge R_0$,
       $$u(x)\ge |x|^{\frac{sp-n}{p-1}},$$
 \end{enumerate}
where ${\mathbb B}_r:=\{x\in\rn:\ |x|< r\}$ and $\overline{\mathbb B}_r:=\{x\in\rn:\ |x|\le r\}$ for any $r\in(0,\infty)$.
\end{theorem}

\begin{proof}
  According to Lemma \ref{lem-SHSL}(ii) (see also \cite[Theorem~9]{KorvenpaaKuusiPalatucci2017MathAnn}), the weak $(s,p)$-supersolution $u$ has a lower semicontinuous representative $\tilde u$ in $\rn$, that is,
  $$
  u(x)=\tilde u(x)=\liminf_{\gfz{y\in\rn}{y\to x}} u(y) \quad \text{for a.e.}\  x\in\rn.
  $$
 For simplicity, we may as well assume that $u$ itself is lower semicontinuous on $\rn$. Then, by Lemma \ref{lem-SHSL}(ii), we know that $u$ is  $(s,p)$-superharmonic  in $\rn$ and, for any $r\in (0,\infty)$,
\begin{align}\label{eq-inf}
\einf_{\mathbb B_r} u=\inf_{\mathbb B_r} u.
\end{align}
  With this,  we are about to validate (i) and (ii) in turn.

 \medskip

{\bf  Step 1:\, proof of (i).\,}
 Let us first show that there exists some $r_0\in(1,\infty)$ satisfying
\begin{align}\label{eq-ur0}
\inf_{\overline{\mathbb B}_{r_0}\setminus {\mathbb B}_1} u\neq 0.
\end{align}
If, for all $r\in(1,\infty)$, we have
  $$f(r):=\inf_{\overline{\mathbb B}_{r}\setminus {\mathbb B}_1} u= 0,$$
then  by using \eqref{eq-WH} and \eqref{eq-inf} we obtain that for some $t$ satisfying \eqref{eq-bar-t}, it holds that
  \begin{align}\label{eq-WH-x1}
\left(\fint_{{\mathbb B}_r} u(x)^t\, dx\right)^\frac 1t
\le C \einf_{{\mathbb B}_{2r}}u
= C \inf_{{\mathbb B}_{2r}}u  \le  C\inf_{\overline{\mathbb B}_r\setminus {\mathbb B}_1} u =0,
\end{align}
which implies that $u=0$  almost everywhere on ${\mathbb B}_r$ and, hence, $u=0$ almost everywhere on $\rn$ via letting $r\to \infty$. This contradicts to the assumption that $u$ is  not  almost everywhere $0$ on $\rn$. Thus, there exists some $r_0\in(1,\infty)$ such that \eqref{eq-ur0} holds.
The previous argument, along with  \eqref{eq-ur0}, further implies that for all $r\ge r_0$,
\begin{align*}
f(r)=\inf_{\overline{\mathbb B}_{r}\setminus {\mathbb B}_1} u\neq 0.
\end{align*}
Indeed, if $f(r)=0$ for some $r>r_0$, then we have by \eqref{eq-WH-x1} that $u=0$ almost everywhere on the open ball ${\mathbb B}_r$ and, hence, $f(r_0)=0$, which is a contradiction!

For such $r_0$ and $r\in(r_0, \infty)$, it is obvious that
$$
u(x)\ge \inf_{\overline{\mathbb B}_{r}\setminus B_1} u\qquad \text{for all}\ \ x\in\overline{\mathbb B}_{r}\setminus {\mathbb B}_1,
$$
which implies
\begin{align}\label{eq1-u>v}
u(x)\ge \left(\inf_{\overline{\mathbb B}_{r}\setminus {\mathbb B}_1} u\right) \left(1-|x|^{\frac{sp-n}{p-1}}\right)
\qquad \text{for all}\ \ x\in\overline{\mathbb B}_{r}
\end{align}
since $u\ge 0$ and the right hand side of \eqref{eq1-u>v} is negative when $x\in {\mathbb B}_1$.
Denote by $v$ the function in the right hand side of \eqref{eq1-u>v}, which is
$(s,p)$-harmonic in $\rn\setminus\{0\}$ by terms of Proposition \ref{prop-harmonic}.
Since $u$ is lower continuous on $\rn$ and $v$ is continuous on $\rn\setminus\{0\}$, it follows that inequality \eqref{eq1-CP} now reads off
\begin{align*}
\liminf_{\gfz{y\in\Omega}{y\to x}} u(y)\ge \liminf_{\gfz{y\in\rn}{y\to x}} u(y) =u(x)\ge v(x)=\limsup_{\gfz{y\in\Omega}{y\to x}} v(y) \quad \text{for all}\  x\in\partial {\mathbb B}_{r}=\partial\Omega,
\end{align*}
where $\Omega:=\rn\setminus {\overline{\mathbb B}_{r}}$.
In view of the above arguments, applying Lemma \ref{lem-KKP-CP-general} yields
\begin{align}\label{eq11-u>v}
u(x)\ge v(x)=\left(\inf_{\overline{\mathbb B}_{r}\setminus {\mathbb B}_1} u\right) \left(1-|x|^{\frac{sp-n}{p-1}}\right)\qquad \text{for a.e.}\ \ x\in\rn.
\end{align}
Thus, we obtain (i).

\medskip

{\bf  Step 2:\, Proof of (ii).}
Via taking $r=r_0$ in \eqref{eq11-u>v} and setting
$$
c_u:= \left(\inf_{\overline{\mathbb B}_{r_0}\setminus {\mathbb B}_1} u\right) \left(1-2^{\frac{sp-n}{p-1}}\right),
$$
which is a positive number,
we then deduce from \eqref{eq11-u>v} that
\begin{align}\label{eq12-u>v}
u \ge c_u\quad \text{a.e. on}\ \ \rn\setminus\mathbb B_2.
\end{align}

Since $r_0>1$ and $sp<n$, there exists a large number $R_0\in (\max\{4,\,r_0\},\,\infty)$ such that
$$R_0^{\frac{sp-n}{p-1}}\le c_u.$$
Suppose that $\epsilon\in(0,\,(4R_0)^{-1})$.
Choose a radial  function $\chi_\epsilon\in C_c^\infty(\rn)$ such that
 $0\le\chi_\epsilon\le 1$ and
$$
\chi_\epsilon(x)=\begin{cases}
1\quad\text{as} \ \ \epsilon< |x|< \epsilon^{-1};\\
0 \quad \text{as} \ \ |x|\le \epsilon/2 \ \text{or }\  |x|\ge 2\epsilon^{-1}.
\end{cases}
$$
For any $x\in\rn$, define
\begin{align*}
\Phi_\epsilon(x):= |x|^{\frac{sp-n}{p-1}} \chi_\epsilon(x),
\end{align*}
which belongs to $C_c^\infty(\rn)$ and, hence, to $W^{s,p}(\rn).$ Note that for each ball $B$ in $\rn$, the  set $\rn\setminus B$  satisfies the measure density condition \eqref{eq-mdc}.
With these facts, we apply Proposition \ref{prop-DP}
and obtain that the  Dirichlet problem
\begin{align}\label{eq-DP}
\begin{cases}
(-\Delta)^s_p w=0 &\quad\text{weakly in}\ {\mathbb B}_{R_0};\\
w=\Phi_\epsilon  &\quad\text{pointwisely on}\ \rn\setminus {\mathbb B}_{R_0},
\end{cases}
\end{align}
has a continuous solution $w\in W^{s,p}(\rn)$.  From
Lemma \ref{lem-SHSL}(iii), it follows that $w$ is $(s,p)$-harmonic  in ${\mathbb B}_{R_0}$.
If $x\in \rn\setminus {\mathbb B}_{R_0}$, then by the second formula of   \eqref{eq-DP}, the definition of $\Phi_\epsilon$ and $0\le \chi_\epsilon\le 1$, one has
$$
w(x)=\Phi_\epsilon(x)= |x|^{\frac{sp-n}{p-1}} \chi_\epsilon(x)\le |x|^{\frac{sp-n}{p-1}} \le R_0^{\frac{sp-n}{p-1}}\le c_u.
$$
In particular,  the continuity of $w$ on $\rn$ ensures
\begin{align*}
\limsup_{\gfz{y\in {\mathbb B}_{R_0}}{y\to x}} w(y)
= \lim_{\gfz{y\in {\mathbb B}_{R_0}}{y\to x}} w(y)=w(x) \le c_u\quad \text{for all}\ x\in \partial {\mathbb B}_{R_0}.
\end{align*}
Invoking the last two formulae, together with the fact that the constant function $c_u$ is $(s,p)$-harmonic in ${\mathbb B}_{R_0}$,
we then use Lemma  \ref{lem-KKP-CP-general} to derive that
\begin{align}\label{eq13-u>v}
w\le c_u\quad \text{a.e. on } \ \rn.
\end{align}

Now, an combination of \eqref{eq12-u>v} and \eqref{eq13-u>v} gives
\begin{align*}
u\ge w\quad \text{a.e. on} \ \ \rn\setminus {\mathbb B}_{2}.
\end{align*}
Meanwhile, note that $0\le u\in W_\loc^{s,p}(\rn), w\in W^{s,p}(\rn)$ and, by the first formula of   \eqref{eq-DP}, we have
\begin{align*}
\langle (-\Delta)^s_p w,\,\phi\rangle = 0 \le \langle (-\Delta)^s_p u,\,\phi\rangle\quad \text{for all}\ 0\le\phi\in C_c^\infty({\mathbb B}_{R_0}).
\end{align*}
Then, applying Lemma \ref{lem-KKP-CP2}  derives
\begin{align}\label{eq-uw}
u\ge w\quad \text{a.e. on} \ \rn.
\end{align}

In particular, by \eqref{eq-uw} and the second formula of  \eqref{eq-DP},
we see that for almost all $x$ satisfying $R_0\le |x|\le \epsilon^{-1}$,
$$
u(x)\ge w(x)=\Phi_\epsilon(x)= |x|^{\frac{sp-n}{p-1}} \chi_\epsilon(x) = |x|^{\frac{sp-n}{p-1}}.
$$
Via letting $\epsilon\to 0$, we observe that this last estimate remains to be  true for almost all $|x|\ge R_0$,
which implies the desired result of (ii).
\end{proof}

The proof of Theorem \ref{thm-mr} implies the following byproduct, which might has its own interest.

\begin{corollary}\label{cor-mr}
Under the same hypothesis of Theorem \ref{thm-mr}, for any $r\in (0,\infty)$, define the function
\begin{align*}
m(r):=\mathop{\mathrm{essinf}}_{{\mathbb B}_{2r}\setminus \mathbb B_r} u.
\end{align*}
Then,  there exist constants  $R_0\in(0,\infty)$ and $C\in (0,\infty)$, depending on $u$, such that  for all $r\ge R_0$,
$$(2r)^{\frac{sp-n}{p-1}}\le m(r)\le C.$$
\end{corollary}

\begin{proof}
Let $R_0$ be the constant determined in the proof of Theorem \ref{thm-mr}.
As a consequence of Theorem \ref{thm-mr}(ii), we have for all $r\ge R_0$ that
$$
m(r)=\einf_{{\mathbb B}_{2r}\setminus {\mathbb B}_r} u \ge  \einf_{{\mathbb B}_{2r}\setminus {\mathbb B}_r} |x|^{\frac{sp-n}{p-1}}
= (2r)^{\frac{sp-n}{p-1}},
$$
which gives the lower estimate of $m$.

The upper estimate of $m$ is a consequence of \eqref{eq11-u>v}. Indeed, if $r\ge R_0$, then
\begin{align*}
  \einf_{{\mathbb B}_4\setminus {\mathbb B}_2} u \ge \einf_{x\in{\mathbb B}_4\setminus {\mathbb B}_2}
  \left(\left(\inf_{\overline{\mathbb B}_{r}\setminus {\mathbb B}_1} u\right)\left(1-|x|^{\frac{sp-n}{p-1}}\right)\right)=\left(\inf_{\overline{\mathbb B}_{r}\setminus {\mathbb B}_1} u\right)\left(1-2^{\frac{sp-n}{p-1}}\right),
\end{align*}
From this, \eqref{eq-inf} and the Harnack inequality in Lemma \ref{lem-Harnack}, it follows that
\begin{align*}
  \left(\einf_{{\mathbb B}_4\setminus {\mathbb B}_2} u\right)
  \left(1-2^{\frac{sp-n}{p-1}}\right)^{-1}
\ge \inf_{\overline{\mathbb B}_{r}\setminus {\mathbb B}_1} u
  &\ge \inf_{{\mathbb B}_{4r}} u= \einf_{{\mathbb B}_{4r}} u
  \end{align*}
  and
  \begin{align*}
  C\einf_{{\mathbb B}_{4r}} u &\ge \left(\fint_{{\mathbb B}_{2r}} u(x)^t\,dx\right)^{\frac 1t}\\
  &\ge  \left(\frac1{|{\mathbb B}_{2r}|}\int_{{\mathbb B}_{2r}\setminus {\mathbb B}_{r}} u(x)^t\,dx\right)^{\frac 1t}\\
  &\ge \left(\einf_{{\mathbb B}_{2r}\setminus {\mathbb B}_{r}} u\right)
  \left(\frac{|{\mathbb B}_{2r}\setminus {\mathbb B}_{r}|}{|{\mathbb B}_{2r}|}\right)^\frac 1t\\
  &= m(r)(1-2^{-n})^{\frac 1t},
\end{align*}
where $t$ and $C$ are the same constants appearing in Lemma \ref{lem-Harnack}.
Combining  the last two formulae indicates that $m(r)$ is bounded by a uniform constant whenever $r\ge R_0$.
\end{proof}

\subsection{Proof of Theorem \ref{thm-Liouville}(i)-(ii)}\label{ss3.4}

\begin{proof}[Proof of Theorem \ref{thm-Liouville} (i) and (ii)]
Let $\eta\in C_c^\infty(\rn)$ such that
$\supp\eta\subseteq \mathbb B_2$, $\eta\equiv 1$ on $\mathbb B_1$ and $0\le\eta\le 1$.
For any $x\in\rn$ and $r\in(0,\infty)$, define
$$
\phi_r(x):=\eta\left(\frac{x}{ r}\right).
$$
Every $\phi_r$ is nonnegative and belongs to $C_c^\infty(\rn)$. Moreover, $\supp\phi_r\subseteq \mathbb B_{2r}$ and $\phi_r\equiv 1$ on $\mathbb B_r$.

Let $s\in(0,1)$, $p\in(0,1)$ and $q\in(0,\infty)$ such that
$0\le u\in W_\loc^{s,p}(\rn)$ satisfies  \eqref{eq-p-Lapineq}, so that we can apply \eqref{eq-phi>} with test functions $\phi_r$, which gives
\begin{align}\label{eq-add}
\int_{\mathbb B_r} u(x)^q\, dx
&\le\int_\rn u(x)^q\phi_r(x)\, dx\\
&\le \frac 12\int_{\rn}\int_{\rn}\frac{|u(x)-u(y)|^{p-2}(u(x)-u(y))(\phi_r(x)-\phi_r(y))}{|x-y|^{n+sp}}\,dy\,dx \notag\\
&\le  \frac 12 \int_{\rn}\int_{\rn}\frac{|u(x)-u(y)|^{p-1}|\phi_r(x)-\phi_r(y)|}{|x-y|^{n+sp}}\,dy\,dx. \notag
\end{align}
A further combination of  \eqref{eq-mr110} and  \eqref{eq-add} yields
\begin{align}\label{eq-add1}
\int_{\mathbb B_r} u(x)^q\, dx
\ls
r^{n-sp}\left(\einf_{\mathbb B_{2r}} u\right)^{p-1},
\end{align}
with implicit constant independent of $r$. Below we show (i) and (ii) of Theorem \ref{thm-Liouville} by considering the following two cases: $p-1\ge q$ and $p-1< q$.

\medskip

{\bf Case 1:\, $p-1\ge q$.\,} In this case, it follows from \eqref{eq-add1} that
$$
\left(\einf_{\mathbb B_{r}} u\right)^q
\ls \fint_{\mathbb B_r} u(x)^q\, dx
\ls
r^{-sp}\left(\einf_{\mathbb B_{2r}} u\right)^{p-1}
\ls r^{-sp}\left(\einf_{\mathbb B_{r}} u\right)^{p-1}.
$$
If $\einf_{\mathbb B_{r}} u\neq 0$ for some $r\in(1,\infty)$, then
$$
r^{sp}\ls \left(\einf_{\mathbb B_{r}} u\right)^{p-1-q}\ls \left(\einf_{\mathbb B_1} u\right)^{p-1-q}.
$$
Consequently, if there exists an increasing sequence $\{r_j\}_{j\in\nn}$ tends to $\infty$ such that $\einf_{\mathbb B_{r_j}} u\neq 0$ for all $j\in\nn$, then
$$
\infty=\lim_{j\to\infty}r_j^{sp}\ls  \left(\einf_{\mathbb B_1} u\right)^{p-1-q},
$$
which contradicts to the assumption of $u\in W_\loc^{s,p}(\rn)$ (that ensures the locality integrability of $u$).
This illustrates that
$$\sup\left\{r\in(0,\infty):\ \einf_{\mathbb B_{r}} u \neq 0\right\} <\infty.$$
In other words, for some large $R$, it holds that
$$\einf_{\mathbb B_R} u= 0$$
This latter fact, along with  \eqref{eq-add1}, further derives that for any $r\in (R,\infty)$,
$$
\int_{\mathbb B_r} u(x)^q\, dx
\ls
r^{n-sp}\left(\einf_{\mathbb B_{2r}} u\right)^{p-1}
\ls
r^{n-sp}\left(\einf_{\mathbb B_{R}} u\right)^{p-1}=0.
$$
Thus, we obtain that $u=0$ almost everywhere on $\mathbb B_r$ and, hence, on $\rn$ by letting $r\to\infty$.

\medskip

{\bf Case 2:\, $p-1< q$.\,}
Again, using \eqref{eq-add1} yields
$$
\int_{\mathbb B_r} u(x)^q\, dx
\ls
r^{n-sp}\left(\einf_{\mathbb B_{2r}} u\right)^{p-1}
\ls r^{n-sp}\left(\einf_{\mathbb B_{r}} u\right)^{p-1}
\ls r^{n-sp}\left(\fint_{\mathbb B_r} u(x)^q\, dx\right)^{\frac{p-1} q},
$$
that is,
\begin{align}\label{eq-new0}
\left(\int_{\mathbb B_r} u(x)^q\, dx\right)^{1-\frac{p-1} q}
\ls  r^{n-sp-\frac{n(p-1)}{q}}.
\end{align}
Now, we let $r\to\infty$ in both sides of \eqref{eq-new0}.
If,  either $sp\ge n$ or $sp<n$ and $q<\frac{n(p-1)}{n-sp}$, then
$$
\lim_{r\to\infty} r^{n-sp-\frac{n(p-1)}{q}}=0,
$$
which induces
$$
\left(\int_\rn u(x)^q\,dx\right)^{1-\frac{p-1} q}=0,
$$
thereby leading to that $u=0$ almost everywhere on $\rn$.

Consider now the case $sp<n$ and $q=\frac{n(p-1)}{n-sp}$, in which case the inequality \eqref{eq-new0} reads off
$$
\left(\int_{\mathbb B_r} u(x)^q\, dx\right)^{1-\frac{p-1} q}
\ls  r^{n-sp-\frac{n(p-1)}{q}}\simeq 1.
$$
Letting $r\to\infty$ gives $u\in L^q(\rn)$.
With this, we will show that $u=0$ almost everywhere on $\rn$ by using
the method of reduction to absurdity. To this end, we assume on the contrary that $u$ is not  almost everywhere zero on $\rn$, so that Theorem \ref{thm-mr} can be applied to derive that
\begin{align}\label{eq-mru-UL}
u(x)\ge |x|^{\frac{sp-n}{p-1}}\quad \text{for all}\ \ |x|\ge R_0,
\end{align}
where  $R_0\in(0,\infty)$ is a large constant depending on $u$.  Since now $q=\frac{n(p-1)}{n-sp}$, the lower estimate of $u$ in \eqref{eq-mru-UL} implies
\begin{align*}
  \int_{|x|\ge R_0} u(x)^q\,dx \ge  \int_{|x|\ge R_0} |x|^{\frac{q(sp-n)}{p-1}} \,dx= \int_{|x|\ge R_0} |x|^{-n}\,dx=\infty.
\end{align*}
This contradicts to the aforementioned conclusion $u\in L^q(\rn)$. Thus, we obtain that $u=0$ almost everywhere on $\rn$ when $sp<n$ and $q=\frac{n(p-1)}{n-sp}$.

Summarizing all the above arguments, we conclude the proofs of (i) and (ii) of Theorem \ref{thm-Liouville}.
\end{proof}

\section{Existence of positive solutions to the fractional $p$-Laplace inequality}\label{sec5}

The aim of this section is to show Theorem \ref{thm-Liouville}(iii).
In other words, under the case $sp<n$ and $q>\frac{n(p-1)}{n-sp}$, we are required to construct a positive solution  $u\in W_\loc^{s,p}(\rn)$ such that
		\begin{equation}\label{eq-p-Laplace=}
		(-\Delta)^s_pu\ge u^q \qquad \text{pointwisely on}\ \ \rn.
		\end{equation}
The idea is to consider the smooth perturbation of the fundamental solution in \eqref{eq-U}. To this end,
for any $x\in\rn$, we define
\begin{align}\label{eq-uf}
f(x):=(1+|x|^2)^\frac{sp-n+\sigma}{2(p-1)},
\end{align}
where $\sigma>0$ is a small value satisfying
\begin{align}\label{eq-sigma}
\begin{cases}
\sigma(p-1)+q(n-sp-\sigma)>n(p-1);\\
2\sigma+sp<n.
\end{cases}
\end{align}
It is easy to check that $f\in W_\loc^{s,p}(\rn)$.
The approach is to show that the smooth function $f$ satisfies \eqref{eq-p-Laplace=} at infinity, thereby getting that the desired solution $u$ will be a positive constant multiple of  $f$.

\subsection{Continuity and positivity of $(-\Delta)^s_p f$}\label{ss5.2}

Throughout this section, we let $f$ be the function defined as in \eqref{eq-uf}, where $\sigma$ is  a small positive number satisfying \eqref{eq-sigma}. The main aim of this section is to show that $(-\Delta)^s_p f$ is a continuous positive function on $\rn$.
Since  $f$ is a radial function, with a little abuse of notation, we will occasionally write $f(x)$ in \eqref{eq-uf} as $f(r)$ whenever  $|x|=r$.

\begin{lemma}
 Let  $s\in(0,1)$ and $p\in (1,\infty)$ such that $sp<n$.
  Assume that $f$ is defined as in \eqref{eq-uf}, where $\sigma$ is  a positive number satisfying \eqref{eq-sigma}.
Then, for any $t\in(0,1)$ and $r\in(0,\infty)$,
\begin{align}\label{eq-f1}
f(rt)>f(r)>f(rt^{-1})
\end{align}
and
\begin{align}\label{eq-f2}
  t^{sp-n}[f(r)-f(rt^{-1})]^{p-1}>[f(rt)-f(r)]^{p-1}.
\end{align}
\end{lemma}

\begin{proof}
From the radially decreasing property of $f$, it follows  that \eqref{eq-f1} holds for all $t\in(0,1)$ and $r\in(0,\infty)$. Thus, in \eqref{eq-f2}, it makes sense to write $[f(r)-f(rt^{-1})]^{p-1}$ and $[f(rt)-f(r)]^{p-1}$.
Observe that \eqref{eq-f2} holds if and only if
$$t^{\frac{sp-n}{p-1}} \left(1-\frac{f(rt^{-1})}{f(r)}\right)>\frac{f(rt)}{f(r)}-1.$$
According to  \eqref{eq-uf}, we have
$$f(r)=(1+r^2)^\frac{sp-n+\sigma}{2(p-1)}.$$
In other words, the inequality \eqref{eq-f2} holds if and only if
$$
t^{\frac{sp-n}{p-1}} \left(1-\left(\frac{1+r^2t^{-2}}{1+r^2}\right)^\frac{sp-n+\sigma}{2(p-1)}\right)
>\left(\frac{1+r^2t^2}{1+r^2}\right)^\frac{sp-n+\sigma}{2(p-1)}-1.
$$
Let $\lambda:=\frac 1{1+r^2}$. Then $1-\lambda=\frac {r^2}{1+r^2}.$
In this way, the proof of \eqref{eq-f2} reduces to verifying that for every  $t\in (0,1)$, the function
\begin{align*}
\varphi_t(\lambda)
&:= t^{\frac{sp-n}{p-1}}  \left(1-\left(\lambda +(1-\lambda) t^{-2}\right)^\frac{sp-n+\sigma}{2(p-1)}\right)
- \left(\lambda+(1-\lambda) t^{2}\right)^\frac{sp-n+\sigma}{2(p-1)}+1\\
&=t^{-\frac{\sigma}{p-1}}  \left(t^\frac{sp-n+\sigma}{p-1}-\left(\lambda t^{2}+1-\lambda \right)^\frac{sp-n+\sigma}{2(p-1)}\right)
- \left(\lambda+(1-\lambda) t^{2}\right)^\frac{sp-n+\sigma}{2(p-1)}+1
\end{align*}
is strictly positive for all $\lambda\in(0,1)$.

Fix $t\in(0,1)$. A straightforward calculation shows that for all $\lambda\in(0,1)$,
\begin{align*}
  \frac{d\varphi_t(\lambda)}{d\lambda}= c_{s,p,n,\sigma}(1-t^2)
  \left(
   \left(\lambda+(1-\lambda) t^{2}\right)^{\frac{sp-n+\sigma}{2(p-1)}-1}
   -t^{-\frac{\sigma}{p-1}}   \left(\lambda t^{2}+1-\lambda \right)^{\frac{sp-n+\sigma}{2(p-1)}-1}
   \right),
\end{align*}
where
$$c_{s,p,n,\sigma}=\frac{n-sp-\sigma}{p-1}.$$
Suppose that $\lambda_\ast\in(0,1)$ is the extreme point of $\varphi_t$, that is,
\begin{align*}
&\frac{d\varphi_t(\lambda)}{d\lambda}\bigg|_{\lambda=\lambda_\ast}=0\\
 &\ \ \Leftrightarrow \ \
   \left(\lambda_\ast+(1-\lambda_\ast) t^{2}\right)^{\frac{sp-n+\sigma}{2(p-1)}-1}
  =t^{-\frac{\sigma}{p-1}}   \left(\lambda_\ast t^{2}+1-\lambda_\ast \right)^{\frac{sp-n+\sigma}{2(p-1)}-1}\\
  &\ \ \Leftrightarrow \ \
  \left(\frac{\lambda_\ast(1-t^2)+t^2}{1-\lambda_\ast(1-t^2)}\right)^{\frac{n-sp-\sigma}{2(p-1)}+1} = t^{\frac{\sigma}{p-1}}\\
   &\ \ \Leftrightarrow \  \ \lambda_\ast = \left(1+ t^\frac{2\sigma}{n-sp-\sigma+2(p-1)}\right)^{-1}
   \left(\frac{t^\frac{2\sigma}{n-sp-\sigma+2(p-1)}-t^2}{1-t^2}\right)
\end{align*}
Due to the second assumption of \eqref{eq-sigma}, we see that $\lambda_\ast\in(0,1)$. But, the function
$$\lambda\mapsto \left(\frac{\lambda(1-t^2)+t^2}{1-\lambda(1-t^2)}\right)^{\frac{n-sp-\sigma}{2(p-1)}+1}$$
is strictly increasing on $(0,1)$, which implies that
$$\lambda<\lambda_\ast\ \ \ \Leftrightarrow \ \ \
  \left(\frac{\lambda(1-t^2)+t^2}{1-\lambda(1-t^2)}\right)^{\frac{n-sp-\sigma}{2(p-1)}+1} < t^{\frac{\sigma}{p-1}} \ \ \ \Leftrightarrow \ \  \ \frac{d\varphi_t(\lambda)}{d\lambda}>0.$$
  Thus $\varphi_t$ is strictly increasing on $(0,\lambda_\ast)$ and strictly decreasing on $(\lambda_\ast, 1)$. Combining the facts $$\varphi_t(1)=0$$ and
$$\varphi_t(0)=\left(t^\frac{sp-n+\sigma}{p-1}-1\right)\left(t^{-\frac{\sigma}{p-1}}-1\right)>0,$$
we conclude that $\varphi_t(\lambda)>0$ for all $\lambda\in(0,1)$. This completes the proof of
\eqref{eq-f2}.
\end{proof}

\begin{lemma}\label{lem-Deltaspf>0}
   Let  $s\in(0,1)$ and $p\in (1,\infty)$ such that $sp<n$.
  Assume that $f$ is defined as in \eqref{eq-uf}, where $\sigma$ is  a positive number satisfying \eqref{eq-sigma}. Then, $(-\Delta)^s_p f$ is a continuous positive function on $\rn$.
\end{lemma}

\begin{proof}
Since $f\in C^\infty(\rn)$, it follows from Proposition \ref{prop-ChenLi} that $(-\Delta)^s_pf$ is  continuous on $\rn$. So, it remains to show the positivity of $(-\Delta)^s_p f$ on $\rn$.

At the point $x=0$, since for all $y\in\rn\setminus\{0\}$, $$f(0)-f(y)=1-(1+|y|^2)^{\frac{sp-n+\sigma}{2(p-1)}}>0,$$ it is obvious that
\begin{align*}
  (-\Delta)^s_pf(0)
  &=\mathrm{p.v.}\int_{\rn}\frac{|f(0)-f(y)|^{p-2}(f(0)-f(y))}{|y|^{n+sp}}\,dy>0.
  \end{align*}

Consider now the case $x\in\rn$ but $x\neq 0$. Set $r=|x|$ and write $x=rA_x{\bf e}_1$, where $A_x$ is a rotation matrix in $\rn$.
Then, by changing of variables $y=rA_x z$ and $z=\rho\theta$, we write
\begin{align*}
  (-\Delta)^s_pf(x)
  &=\mathrm{p.v.}\int_{\rn}\frac{|f(x)-f(y)|^{p-2}(f(x)-f(y))}{|x-y|^{n+sp}}\,dy\\
  &=r^{-sp}\mathrm{p.v.}\int_{\rn}\frac{|f(r)-f(r|z|)|^{p-2}(f(r)-f(r|z|))}{|{\bf e}_1-z|^{n+sp}}\,dz\\
  &=r^{-sp}\lim_{\epsilon\to 0}\int_0^\infty\int_{\gfz{\theta\in\mathbb S^{n-1}}{|{\bf e}_1-\rho\theta|>\epsilon}}\frac{|f(r)-f(r\rho)|^{p-2}(f(r)-f(r\rho))}{|{\bf e}_1-\rho\theta|^{n+sp}}\,\rho^{n-1}\,d\sigma(\theta)\,d\rho
\end{align*}
So, if $r=|x|\neq 0$, then the proof of $(-\Delta)^s_pf(x)>0$ falls into verifying that
\begin{align}\label{eq-Delta-fr}
\lim_{\epsilon\to 0} \int_0^\infty\int_{\gfz{\theta\in\mathbb S^{n-1}}{|{\bf e}_1-\rho\theta|>\epsilon}}\frac{|f(r)-f(r\rho)|^{p-2}(f(r)-f(r\rho))}
{|{\bf e}_1-\rho\theta|^{n+sp}}\,\rho^{n-1}\,d\sigma(\theta)\,d\rho>0.
\end{align}

Below, let us show \eqref{eq-Delta-fr}. Fix $\epsilon\in(0,1)$. On the one hand,  by a change of variables $\rho=\frac 1t$ and observing that $|\theta|=|{\bf e}_1|=1$ implies
\begin{align}\label{eq-te1theta}
|t{\bf e}_1-\theta|^2=t^2+1-2t{\bf e}_1\cdot\theta=|{\bf e}_1-t\theta|^2,
\end{align}we then derive that
\begin{align*}
 & \int_1^\infty\int_{\gfz{\theta\in\mathbb S^{n-1}}{|{\bf e}_1-\rho\theta|>\epsilon}}\frac{|f(r)-f(r\rho)|^{p-2}(f(r)-f(r\rho))}{|{\bf e}_1-\rho\theta|^{n+sp}}\,\rho^{n-1}\,d\sigma(\theta)\,d\rho\\
 &\quad = \int_0^1 \int_{\gfz{\theta\in\mathbb S^{n-1}}{|t{\bf e}_1-\theta|>t\epsilon}}
 \frac{|f(r)-f(rt^{-1})|^{p-2}(f(r)-f(rt^{-1}))}{|{\bf e}_1-t^{-1}\theta|^{n+sp}}\,t^{-n-1}\,d\sigma(\theta)\,dt\\
 &\quad = \int_0^1 \int_{\gfz{\theta\in\mathbb S^{n-1}}{|{\bf e}_1- t\theta|>t\epsilon}}
 \frac{t^{sp-n}|f(r)-f(rt^{-1})|^{p-2}(f(r)-f(rt^{-1}))}{|{\bf e}_1-t\theta|^{n+sp}}\,t^{n-1}\,d\sigma(\theta)\,dt\\
 &\quad = \int_0^1 \int_{\gfz{\theta\in\mathbb S^{n-1}}{|{\bf e}_1- t\theta|>t\epsilon}}
 \frac{t^{sp-n}[f(r)-f(rt^{-1})]^{p-1}}{|{\bf e}_1-t\theta|^{n+sp}}\,t^{n-1}\,d\sigma(\theta)\,dt,
\end{align*}
where the last step follows from \eqref{eq-f1}.
On the other hand, from \eqref{eq-f1}, it also follows that
\begin{align*}
 &\int_0^1\int_{\gfz{\theta\in\mathbb S^{n-1}}{|{\bf e}_1-\rho\theta|>\epsilon}}\frac{|f(r)-f(r\rho)|^{p-2}(f(r)-f(r\rho))}
{|{\bf e}_1-\rho\theta|^{n+sp}}\,\rho^{n-1}\,d\sigma(\theta)\,d\rho \\
&\quad = - \int_0^1\int_{\gfz{\theta\in\mathbb S^{n-1}}{|{\bf e}_1-\rho\theta|>\epsilon}}\frac{[f(r\rho)-f(r)]^{p-1}}
{|{\bf e}_1-\rho\theta|^{n+sp}}\,\rho^{n-1}\,d\sigma(\theta)\,d\rho.
\end{align*}
 Combining the last two formulae gives
\begin{align}\label{eq-f01}
  &\int_0^\infty\int_{\gfz{\theta\in\mathbb S^{n-1}}{|{\bf e}_1-\rho\theta|>\epsilon}}\frac{|f(r)-f(r\rho)|^{p-2}(f(r)-f(r\rho))}
{|{\bf e}_1-\rho\theta|^{n+sp}}\,\rho^{n-1}\,d\sigma(\theta)\,d\rho\\
&\quad =\int_0^1 \int_{\gfz{\theta\in\mathbb S^{n-1}}{\epsilon\ge |{\bf e}_1- t\theta|>t\epsilon}}
 \frac{t^{sp-n}[f(r)-f(rt^{-1})]^{p-1}}{|{\bf e}_1-t\theta|^{n+sp}}\,t^{n-1}\,d\sigma(\theta)\,dt\notag\\
 &\qquad + \int_0^1\int_{\gfz{\theta\in\mathbb S^{n-1}}{|{\bf e}_1-t\theta|>\epsilon}}\frac{t^{sp-n}[f(r)-f(rt^{-1})]^{p-1}-[f(rt)-f(r)]^{p-1}}
{|{\bf e}_1-t\theta|^{n+sp}}\,t^{n-1}\,d\sigma(\theta)\,dt.\notag
\end{align}
Due to \eqref{eq-f1} and \eqref{eq-f2}, both the integrands in the right hand side of \eqref{eq-f01} are nonnegative. Moreover, if $\epsilon<1/2$, then the second double-integral in  the right hand side of \eqref{eq-f01} is greater than
\begin{align*}
  \int_0^{1/2}\int_{\theta\in\mathbb S^{n-1}}\frac{t^{sp-n}[f(r)-f(rt^{-1})]^{p-1}-[f(rt)-f(r)]^{p-1}}
{|{\bf e}_1-t\theta|^{n+sp}}\,t^{n-1}\,d\sigma(\theta)\,dt,
\end{align*}
which is a positive value. This proves  \eqref{eq-Delta-fr} and, hence,  $(-\Delta)^s_pf(x)>0$ for all $x\neq 0$. Thus, we complete the proof of Lemma \ref{lem-Deltaspf>0}.
\end{proof}

\subsection{An approximation lemma}\label{ss5.1}

 The main aim of
 this section is to prove the following approximation lemma.

\begin{lemma}\label{lem-approx}
  Let  $s\in(0,1)$ and $p\in (1,\infty)$ such that $sp<n$.
  Choose a small positive number $\sigma$ satisfying \eqref{eq-sigma}.
  For any $r\in(0,\infty)$ and $z\in\rn$, set
\begin{align}\label{eq-gr}
g_r(z):=\left(\frac 1{r^2}+|z|^2\right)^{\frac{sp-n+\sigma}{2(p-1)}}
\end{align}
and
\begin{align}\label{eq-Gr0}
  G(r):=\mathrm{p.v.}\int_\rn \frac{|g_r({\bf e}_1)-g_r(z)|^{p-2}(g_r({\bf e}_1)-g_r(z))}
  {|{\bf e}_1-z|^{n+sp}}\,dz.
\end{align}
Then,
\begin{align*}
  \lim_{r\to +\infty}G(r)=\mathrm{p.v.}\int_\rn \frac{|1-|z|^{\frac{sp-n+\sigma}{p-1}}|^{p-2}(1-|z|^{\frac{sp-n+\sigma}{p-1}})}
  {|{\bf e}_1-z|^{n+sp}}\,dz,
\end{align*}
where the principle value integral in the right hand side converges to a positive finite  number.
\end{lemma}

\begin{proof}
According to Proposition \ref{prop-ChenLi}, the principle value integral in defining $G(r)$ converges for all $r\in(0,\infty)$.
Let us split the integral domain of $G(r)$ into three parts:
\begin{align*}
\begin{cases}
  \Omega_1: = \{z\in\rn:\, |{\bf e}_1-z|<1/2\};\\
  \Omega_2: = \{z\in\rn:\, 1/2\le|{\bf e}_1-z|\le 2\};\\
  \Omega_3: = \{z\in\rn:\, |{\bf e}_1-z|> 2\}.
    \end{cases}
\end{align*}
As a consequence, we write
\begin{align}\label{eq-Gr}
  G(r) &= \left(\mathrm{p.v.}\int_{{\Omega_1}}+ \int_{{\Omega_2}}+\int_{{\Omega_3}}\right)\frac{|g_r({\bf e}_1)-g_r(z)|^{p-2}(g_r({\bf e}_1)-g_r(z))}
  {|{\bf e}_1-z|^{n+sp}}\,dz.
\end{align}
Since we will consider the limit of $G(r)$ as $r\to\infty$, we may as well assume that $r\ge 1$ in what follows. The subsequent arguments is split into four steps.

\medskip

{\bf  Step 1:\, limit of the principle integral over the domain $\Omega_1$.} As in \eqref{eq-h}, we again use the notation
\begin{align*}
h(t):=|t|^{p-2}t\quad \text{for all}\ \, t\in\rr,
\end{align*}
so that
$$
|g_r({\bf e}_1)-g_r(z)|^{p-2}(g_r({\bf e}_1)-g_r(z))= h\big(g_r({\bf e}_1)-g_r(z)\big)
$$
and
$$
|\big(z-{\bf e}_1\big)\cdot \nabla g_r({\bf e}_1)|^{p-2}\big(z-{\bf e}_1\big)\cdot \nabla g_r({\bf e}_1)= h\big(\big(z-{\bf e}_1\big)\cdot \nabla g_r({\bf e}_1)\big).
$$
 Due to symmetry, it holds that
$$
\mathrm{p.v.}\int_{\Omega_1} \frac{|\big(z-{\bf e}_1\big)\cdot \nabla g_r({\bf e}_1)|^{p-2}\big(z-{\bf e}_1\big)\cdot \nabla g_r({\bf e}_1)}
  {|{\bf e}_1-z|^{n+sp}}\,dz=0,
$$
whence
\begin{align}\label{eq-Omega1}
&\mathrm{p.v.}\int_{{\Omega_1}}\frac{|g_r({\bf e}_1)-g_r(z)|^{p-2}(g_r({\bf e}_1)-g_r(z))}
  {|{\bf e}_1-z|^{n+sp}}\,dz \\
  &\quad=  \int_{{\Omega_1}}
  \frac{h\big(g_r({\bf e}_1)-g_r(z)\big)-h\big(\big(z-{\bf e}_1\big)\cdot \nabla g_r({\bf e}_1)\big)}
  {|{\bf e}_1-z|^{n+sp}}\,dz.\notag
\end{align}
To deal with the numerator
 in the right hand integral of \eqref{eq-Omega1}, we apply \eqref{eq-ChenLi57} with
$$A=\big(z-{\bf e}_1\big)\cdot \nabla g_r({\bf e}_1)$$
and
$$B=g_r(z)-g_r({\bf e}_1)-(z-{\bf e}_1)\cdot \nabla g_r({\bf e}_1).$$
A straightforward calculation shows that for all $z=(z_1,\dots, z_n)\in\rn$ and $i,j\in\{1,2,\dots,n\}$,
$$
\partial_i g_r(z)=-c_{s,p,n,\sigma}\, z_i \left(\frac 1{r^2}+|z|^2\right)^{\frac{sp-n+\sigma}{2(p-1)}-1}
$$
and
$$
\partial_j\partial_i g_r(z)=
c_{s,p,n,\sigma}
\left[\left(c_{s,p,n,\sigma}+2\right) z_iz_j
-\left(\frac 1{r^2}+|z|^2\right)\delta_{ij}\right]\left(\frac 1{r^2}+|z|^2\right)^{\frac{sp-n+\sigma}{2(p-1)}-2},
$$
where
$$c_{s,p,n,\sigma}=\frac{n-sp-\sigma}{p-1}$$
and
$$
\delta_{ij}=
\begin{cases}
 0  & \mbox{as}\ \, i\neq j;\\
 1& \mbox{as}\ \, i=j.
\end{cases}
$$
Since $r\in(1,\infty)$, it follows that $|\nabla g_r({\bf e}_1)|\simeq 1$, which in turn gives
\begin{align}\label{eq-A}
|A|\le|z-{\bf e}_1|\left|\nabla g_r({\bf e}_1)\right|\ls |z-{\bf e}_1|.
\end{align}
Meanwhile,  by using the Taylor expansion of $g_r$ at the point ${\bf e}_1$, we then have
$$
B=g_r(z)-g_r({\bf e}_1)-(z-{\bf e}_1)\cdot \nabla g_r({\bf e}_1) =\frac 12 \sum_{i=1}^n\sum_{j=1}^n (z_i-{\bf e}_{1i})(z_j-{\bf e}_{1j}) \partial_{i}\partial_{j}g_r(\xi),
$$
where $\xi={\bf e}_1+\theta (z-{\bf e}_1)$ for some $\theta\in (0,1)$ and $z_i, {\bf e}_{1i}$ denote the $i$-th entry of $z$ and ${\bf e}_1$, respectively.
For any $z\in \Omega_1$, we have
$$
 |\xi|\ge |{\bf e}_1|-|z-{\bf e}_1|>1/2
$$
and
$$
|\xi|\le |{\bf e}_1|-|z-{\bf e}_1|<2,
$$
which leads to
$$|\partial_{i}\partial_{j}g_r(\xi)|\le C$$
for some constant $C$ depending only on $s,p,n$ and $\sigma$. This further implies
\begin{align}\label{eq-B}
 |B|= \left|g_r(z)-g_r({\bf e}_1)-(z-{\bf e}_1)\cdot \nabla g_r({\bf e}_1)\right|\le \frac {C n^2} 2|z-{\bf e}_1|^2.
\end{align}
From \eqref{eq-ChenLi57}, \eqref{eq-A}, \eqref{eq-B}  and $|z-{\bf e}_1|<1/2$,  it follows that
\begin{align}\label{eq-gr1}
  \left|\frac{h\big(g_r({\bf e}_1)-g_r(z)\big)-h\big(\big(z-{\bf e}_1\big)\cdot \nabla g_r({\bf e}_1)\big)}
  {|{\bf e}_1-z|^{n+sp}}\right|
  & \ls \frac{(|z-{\bf e}_1|+|z-{\bf e}_1|^2)^{p-2} |z-{\bf e}_1|^2}
  {|{\bf e}_1-z|^{n+sp}}\\
  &\ls |z-{\bf e}_1|^{(1-s)p-n} \notag
\end{align}
holds uniformly in $r\in (1,\infty)$ and $z\in\Omega_1$.  Of course,
\begin{align}\label{eq-gr11}
\int_{\Omega_1}|{\bf e}_1-z|^{(1-s)p-n}\,dz=\int_{|w|<1/2} |w|^{(1-s)p-n}\,dw<\infty.
\end{align}
A combination of \eqref{eq-gr1} and \eqref{eq-gr11}
enables us to use
the Lebesgue dominated convergence theorem. So, by \eqref{eq-Omega1}, together with
\begin{align*}
\lim_{r\to\infty}  g_r(z)=|z|^{\frac{sp-n+\sigma}{p-1}}
\end{align*}
and
\begin{align*}
\lim_{r\to\infty} \nabla g_r({\bf e}_1)
= \left(\frac{sp-n+\sigma}{p-1}\right) \left( z \lim_{r\to\infty} \left(\frac 1{r^2}+|z|^2\right)^{\frac{sp-n+\sigma}{2(p-1)}-1}\right)\bigg|_{z={\bf e_1}}
= \left(\frac{sp-n+\sigma}{p-1}\right) {\bf e}_1,
\end{align*}
we
further obtain
\begin{align}\label{eq-Gr1}
&\lim_{r\to\infty}\mathrm{p.v.}\int_{{\Omega_1}}\frac{|g_r({\bf e}_1)-g_r(z)|^{p-2}(g_r({\bf e}_1)-g_r(z))}
  {|{\bf e}_1-z|^{n+sp}}\,dz\\
  &=\lim_{r\to\infty}\int_{\Omega_1}
  \frac{h\big(g_r({\bf e}_1)-g_r(z)\big)-h\big(\big(z-{\bf e}_1\big)\cdot \nabla g_r({\bf e}_1)\big)}
  {|{\bf e}_1-z|^{n+sp}}\,dz\notag\\
  &=\int_{\Omega_1}\lim_{r\to\infty}
  \frac{h\big(g_r({\bf e}_1)-g_r(z)\big)-h\big(\big(z-{\bf e}_1\big)\cdot \nabla g_r({\bf e}_1)\big)}
  {|{\bf e}_1-z|^{n+sp}}\,dz\notag\\
  &=\int_{\Omega_1}\frac{h\big(1-|z|^{\frac{sp-n+\sigma}{p-1}}\big)-h\big(
  (z-{\bf e}_1)\cdot (\frac{sp-n+\sigma}{p-1} {\bf e}_1)
  \big)}
  {|{\bf e}_1-z|^{n+sp}}\,dz\notag\\
  &=\int_{\Omega_1} \frac{|1-|z|^{\frac{sp-n+\sigma}{p-1}}|^{p-2}\big(
  1-|z|^{\frac{sp-n+\sigma}{p-1}}\big)-(\frac{sp-n+\sigma}{p-1})^{p-1}|(z-{\bf e}_1)\cdot {\bf e}_1|^{p-2}(z-{\bf e}_1)\cdot {\bf e}_1}
  {|{\bf e}_1-z|^{n+sp}}\,dz\notag\\
 &=\mathrm{p.v.}\int_{\Omega_1} \frac{|1-|z|^{\frac{sp-n+\sigma}{p-1}}|^{p-2}\big(
  1-|z|^{\frac{sp-n+\sigma}{p-1}}\big)}
  {|{\bf e}_1-z|^{n+sp}}\,dz,\notag
  \end{align}
  where in the last step we used  the fact that
  $$
  \mathrm{p.v.}\int_{\Omega_1}
  \frac{|(z-{\bf e}_1)\cdot {\bf e}_1|^{p-2}(z-{\bf e}_1)\cdot {\bf e}_1}
  {|{\bf e}_1-z|^{n+sp}}\,dz=0.
  $$

\medskip

{\bf  Step 2:\, limit of the integral over the domain $\Omega_2$.}
Note that any $z\in\Omega_2$ satisfies $1/2\le |z-{\bf e}_1|\le 2$ and, hence,
$$
|z|\le |z-{\bf e}_1|+|{\bf e}_1|\le 3.
$$
This, combined with the fact that $sp-n+\sigma<0$ (see  \eqref{eq-sigma}), implies that  for all $r\in(1,\infty)$
and $z\in\Omega_2$,
$$
g_r({\bf e_1}) =\left(\frac 1{r^2}+1\right)^{\frac{sp-n+\sigma}{2(p-1)}}
\le 3^{\frac{n-sp-\sigma}{p-1}} \left(\frac 1{r^2}+|z|^2\right)^{\frac{sp-n+\sigma}{2(p-1)}}
=3^{\frac{n-sp-\sigma}{p-1}} g_r(z).
$$
Thus, for any $z\in\Omega_2$,  we have
\begin{align*}
    \left|\frac{|g_r({\bf e}_1)-g_r(z)|^{p-2}(g_r({\bf e}_1)-g_r(z))}
  {|{\bf e}_1-z|^{n+sp}}\right| \ls g_r(z)^{p-1} \ls |z|^{sp-n+\sigma}.
\end{align*}
Moreover, observe that
\begin{align*}
\int_{\Omega_2}|z|^{sp-n+\sigma}\,dz=\int_{|z|<3} |z|^{sp-n+\sigma}\,dz<\infty.
\end{align*}
Due to the last two estimates, we apply the Lebesgue dominated convergence theorem to derive
\begin{align}\label{eq-Gr2}
&\lim_{r\to\infty}\int_{{\Omega_2}}\frac{|g_r({\bf e}_1)-g_r(z)|^{p-2}(g_r({\bf e}_1)-g_r(z))}
  {|{\bf e}_1-z|^{n+sp}}\,dz\\
  &\quad=\int_{\Omega_2}\lim_{r\to\infty}
  \frac{|g_r({\bf e}_1)-g_r(z)|^{p-2}(g_r({\bf e}_1)-g_r(z))}
  {|{\bf e}_1-z|^{n+sp}}\,dz\notag\\
 &\quad=\int_{\Omega_2} \frac{|1-|z|^{\frac{sp-n+\sigma}{p-1}}|^{p-2}\big(
  1-|z|^{\frac{sp-n+\sigma}{p-1}}\big)}
  {|{\bf e}_1-z|^{n+sp}}\,dz.\notag
  \end{align}

\medskip

{\bf  Step 3:\, limit of the integral over the domain $\Omega_3$.} In this case,  any $z\in\Omega_3$ satisfies
$$|z|\ge |z-{\bf e}_1|-|{\bf e}_1|>1.$$
Moreover, note that
$0<g_r(z)<1$ and $ 0<g_r({\bf e}_1)< 1,$
which gives
\begin{align*}
    \left|\frac{|g_r({\bf e}_1)-g_r(z)|^{p-2}(g_r({\bf e}_1)-g_r(z))}
  {|{\bf e}_1-z|^{n+sp}}\right| \le \frac 1 {|{\bf e}_1-z|^{n+sp}}.
\end{align*}
Further, invoking the fact that
\begin{align*}
\int_{\Omega_3}\frac 1 {|{\bf e}_1-z|^{n+sp}}\,dz=\int_{|w|>2} \frac 1 {|w|^{n+sp}}\,dw<\infty,
\end{align*}
we then apply again the Lebesgue dominated convergence theorem, thereby deriving
\begin{align}\label{eq-Gr3}
&\lim_{r\to\infty}\int_{{\Omega_3}}\frac{|g_r({\bf e}_1)-g_r(z)|^{p-2}(g_r({\bf e}_1)-g_r(z))}
  {|{\bf e}_1-z|^{n+sp}}\,dz\\
  &\quad=\int_{\Omega_3}\lim_{r\to\infty}
  \frac{|g_r({\bf e}_1)-g_r(z)|^{p-2}(g_r({\bf e}_1)-g_r(z))}
  {|{\bf e}_1-z|^{n+sp}}\,dz\notag\\
 &\quad=\int_{\Omega_3} \frac{|1-|z|^{\frac{sp-n+\sigma}{p-1}}|^{p-2}\big(
  1-|z|^{\frac{sp-n+\sigma}{p-1}}\big)}
  {|{\bf e}_1-z|^{n+sp}}\,dz.\notag
  \end{align}

\medskip

{\bf  Step 4:\, positivity of the limit of $G(r)$ when $r\to \infty$.}
Combining \eqref{eq-Gr}-\eqref{eq-Gr1}-\eqref{eq-Gr2}-\eqref{eq-Gr3} yields
\begin{align*}
  \lim_{r\to \infty}G(r)
  =\mathrm{p.v.}\int_\rn \frac{|1-|z|^{\frac{sp-n+\sigma}{p-1}}|^{p-2}(1-|z|^{\frac{sp-n+\sigma}{p-1}})}
  {|{\bf e}_1-z|^{n+sp}}\,dz,
\end{align*}
whose right hand side integral converges to a finite number by terms of \eqref{eq-finite} in
Proposition \ref{prop-ChenLi}. Now, we are left to validate that it is a positive number.
This is almost the same argument as the proof of Lemma \ref{lem-Deltaspf>0}.
Indeed, for any $\ez\in(0,1)$, by using the polar coordinate transformation, we write
\begin{align}\label{eq-xx0}
 &\int_{|{\bf e}_1-z|>\epsilon} \frac{|1-|z|^{\frac{sp-n+\sigma}{p-1}}|^{p-2}(1-|z|^{\frac{sp-n+\sigma}{p-1}})}
  {|{\bf e}_1-z|^{n+sp}}\,dz\\
  &\quad =\int_0^\infty \int_{\gfz{\theta\in\mathbb S^{n-1}}{|{\bf e}_1-\rho\theta|>\epsilon}}
  \frac{|1-\rho^{\frac{sp-n+\sigma}{p-1}}|^{p-2}(1-\rho^{\frac{sp-n+\sigma}{p-1}})}
  {|{\bf e}_1-\rho\theta|^{n+sp}}\rho^{n-1}\,d\sigma(\theta)\,d\rho\notag\\
  &\quad =\left(\int_0^1 \int_{\gfz{\theta\in\mathbb S^{n-1}}{|{\bf e}_1-\rho\theta|>\epsilon}}
  +\int_1^\infty \int_{\gfz{\theta\in\mathbb S^{n-1}}{|{\bf e}_1-\rho\theta|>\epsilon}}\right)
  \frac{|1-\rho^{\frac{sp-n+\sigma}{p-1}}|^{p-2}(1-\rho^{\frac{sp-n+\sigma}{p-1}})}
  {|{\bf e}_1-\rho\theta|^{n+sp}}\rho^{n-1}\,d\sigma(\theta)\,d\rho.\notag
\end{align}
To deal with the second double integral in the right hand side of \eqref{eq-xx0}, note that a change of variables $\rho=1/t$ gives
\begin{align*}
\frac{|1-\rho^{\frac{sp-n+\sigma}{p-1}}|^{p-2}(1-\rho^{\frac{sp-n+\sigma}{p-1}})}
  {|{\bf e}_1-\rho\theta|^{n+sp}}
  &=  \frac{|1-t^{-\frac{sp-n+\sigma}{p-1}}|^{p-2}(1-t^{-\frac{sp-n+\sigma}{p-1}})}
  {|{\bf e}_1-t^{-1}\theta|^{n+sp}}\\
  &= t^{2n-\sigma}\frac{|t^{\frac{sp-n+\sigma}{p-1}}-1|^{p-2}(t^{\frac{sp-n+\sigma}{p-1}}-1)}
  {|t{\bf e}_1-\theta|^{n+sp}}.
\end{align*}
Moreover, invoking the equality $|t{\bf e}_1-\theta|=|{\bf e}_1-t\theta|$ given in \eqref{eq-te1theta},
we then utilize a change of variables $\rho=1/t$ to deduce that
\begin{align*}
  &\int_1^\infty \int_{\gfz{\theta\in\mathbb S^{n-1}}{|{\bf e}_1-\rho\theta|>\epsilon}}
  \frac{|1-\rho^{\frac{sp-n+\sigma}{p-1}}|^{p-2}(1-\rho^{\frac{sp-n+\sigma}{p-1}})}
  {|{\bf e}_1-\rho\theta|^{n+sp}}\rho^{n-1}\,d\sigma(\theta)\,d\rho\\
   &\quad=\int_0^1 \int_{\gfz{\theta\in\mathbb S^{n-1}}{|{\bf e}_1-t^{-1}\theta|>\epsilon}}
  \frac{|t^{\frac{sp-n+\sigma}{p-1}}-1|^{p-2}(t^{\frac{sp-n+\sigma}{p-1}}-1)}
  {|{\bf e}_1-t\theta|^{n+sp}}t^{n-\sigma-1}\,d\sigma(\theta)\,dt.
\end{align*}
Observe that $|{\bf e}_1-t^{-1}\theta|>\epsilon$ in the integral domain can be equivalently written as $|{\bf e}_1-t\theta|>t\epsilon$.
This, combined with  \eqref{eq-xx0}, further induces
\begin{align}\label{eq-xx0xx}
 &\int_{|{\bf e}_1-z|>\epsilon} \frac{|1-|z|^{\frac{sp-n+\sigma}{p-1}}|^{p-2}(1-|z|^{\frac{sp-n+\sigma}{p-1}})}
  {|{\bf e}_1-z|^{n+sp}}\,dz\\
  &\quad = \int_0^1 \int_{\gfz{\theta\in\mathbb S^{n-1}}{\epsilon\ge |{\bf e}_1-t\theta|>t\epsilon}}
  \frac{|t^{\frac{sp-n+\sigma}{p-1}}-1|^{p-2}(t^{\frac{sp-n+\sigma}{p-1}}-1)}
  {|{\bf e}_1-t\theta|^{n+sp}}t^{n-\sigma-1}\,d\sigma(\theta)\,dt\notag\\
  &\qquad + \int_0^1 \int_{\gfz{\theta\in\mathbb S^{n-1}}{|{\bf e}_1-t\theta|>\epsilon}} (t^{-\sigma}-1)
  \frac{|t^{\frac{sp-n+\sigma}{p-1}}-1|^{p-2}(t^{\frac{sp-n+\sigma}{p-1}}-1)}
  {|{\bf e}_1-t\theta|^{n+sp}}t^{n-1}\,d\sigma(\theta)\,dt.\notag
  \end{align}
  Due to $sp-n+\sigma<0$ and $0<t\le 1$, both the two integrals in
  in the right hand side of \eqref{eq-xx0xx} have their integrands being nonnegative, which gives that for any $\epsilon\in(0,1/2)$,
\begin{align*}
  &\int_{|{\bf e}_1-z|>\epsilon} \frac{|1-|z|^{\frac{sp-n+\sigma}{p-1}}|^{p-2}(1-|z|^{\frac{sp-n+\sigma}{p-1}})}
  {|{\bf e}_1-z|^{n+sp}}\,dz\\
  &\quad\ge\int_0^{1/2} \int_{\theta\in\mathbb S^{n-1}} (t^{-\sigma}-1)
  \frac{|t^{\frac{sp-n+\sigma}{p-1}}-1|^{p-2}(t^{\frac{sp-n+\sigma}{p-1}}-1)}
  {|{\bf e}_1-t\theta|^{n+sp}}t^{n-1}\,d\sigma(\theta)\,dt>0,
\end{align*}
  Upon letting $\epsilon\to0$, we obtain
  $$
  \mathrm{p.v.}\int_\rn \frac{|1-|z|^{\frac{sp-n+\sigma}{p-1}}|^{p-2}(1-|z|^{\frac{sp-n+\sigma}{p-1}})}
  {|{\bf e}_1-z|^{n+sp}}\,dz>0,
  $$
  as desired.
\end{proof}

\subsection{Proof of Theorem \ref{thm-Liouville}(iii)}\label{ss5.3}

\begin{proof}[Proof of Theorem \ref{thm-Liouville}(iii)]
Let $f$ be the positive function defined in \eqref{eq-uf}, with the parameter  $\sigma$ therein satisfying \eqref{eq-sigma}. Obviously, it holds that $f\in W_\loc^{s,p}(\rn)$.
Our first aim is to show that
\begin{align}\label{eq-aim}
\lim_{|x|\to\infty} \frac{(-\Delta)^s_pf(x)}{f(x)^q}=+\infty.
\end{align}

Fix $x\in\rn$ such that $x\neq 0$. Set $r=|x|$ and write $x=rA_x{\bf e}_1$, where $A_x$ is a rotation matrix in $\rn$.
Upon writing $y=rA_x z$, we have
\begin{align*}
f(x)-f(y)= r^{\frac{sp-n+\sigma}{p-1}} \left(
\left(\frac 1{r^2}+1\right)^{\frac{sp-n+\sigma}{2(p-1)}}
  -\left(\frac 1{r^2}+|z|^2\right)^{\frac{sp-n+\sigma}{2(p-1)}}\right).
\end{align*}
To simplify the notation, for any $r\in(0,\infty)$, let the function $g_r$ be  as in \eqref{eq-gr}, that is,
$$
g_r(z)=\left(\frac 1{r^2}+|z|^2\right)^{\frac{sp-n+\sigma}{2(p-1)}}\quad \text{for all}\ z\in\rn.
$$
Consequently, we obtain
\begin{align*}
|f(x)-f(y)|^{p-2}(f(x)-f(y))
= r^{sp-n+\sigma} |g_r({\bf e}_1)-g_r(z)|^{p-2}(g_r({\bf e}_1)-g_r(z)),
  \end{align*}
which further implies
\begin{align*}
  (-\Delta)^s_pf(x)
  &=\mathrm{p.v.}\int_{\rn}\frac{|f(x)-f(y)|^{p-2}(f(x)-f(y))}{|x-y|^{n+sp}}\,dy\\
  &=r^{\sigma-n}\, \mathrm{p.v.}
  \int_\rn \frac{|g_r({\bf e}_1)-g_r(z)|^{p-2}(g_r({\bf e}_1)-g_r(z))}
  {|{\bf e}_1-z|^{n+sp}}\,dz\\
  &=r^{\sigma-n}G(r),
\end{align*}
with $G(r)$ being defined as in \eqref{eq-Gr0}.
Now,
\begin{align}\label{eq-aim1}
  \lim_{|x|\to\infty} \frac{(-\Delta)^s_pf(x)}{f(x)^q}
  &=
  \lim_{r\to\infty} \frac{r^{\sigma-n} G(r)}{(1+r^2)^{\frac{q(sp-n+\sigma)}{2(p-1)}}}=
  \lim_{r\to\infty} r^{\sigma-n-\frac{q(sp-n+\sigma)}{p-1}} G(r).
\end{align}
As was proved in Lemma \ref{lem-approx}, if $r\to \infty$, then $G(r)$
converges to a positive real number. However, the choice of $\sigma$ in the first condition of \eqref{eq-sigma} implies that
$$
 \lim_{r\to\infty} r^{\sigma-n-\frac{q(sp-n+\sigma)}{p-1}}=+\infty.
$$
 From this and \eqref{eq-aim1}, it follows that \eqref{eq-aim} holds.

By \eqref{eq-aim}, there exists a large positive number $M$ such that when  $|x|\ge M$,
$$
(-\Delta)^s_pf(x)\ge f(x)^q.
$$
Based on Lemma \ref{lem-Deltaspf>0}, the function $(-\Delta)^s_p f$ is a continuous positive function on $\rn$.
Invoking the positivity and continuity of $f^q$ on $\rn$, it makes sense to define
$$
c_f:=\max\left\{1,\ \max_{|x|\le M}\frac{f(x)^q}{(-\Delta)^s_pf(x)}\right\}.
$$
The previous construction of $f$ implies that
$$
 f(x)^q \le c_f\, (-\Delta)^s_pf(x)
 \ \ \text{ for all }\ x\in\rn.
$$
Now, letting
$$u:= c_f^{\frac 1{p-q-1}} f,$$
we see that such $u$ is a positive  pointwise solution to
$$
 (-\Delta)^s_pu(x)\ge u(x)^q
 \ \ \text{ for all }\ x\in\rn.
$$
This completes the proof of Theorem \ref{thm-Liouville}(iii).
\end{proof}

\medskip

\noindent{\bf Conflict of Interest and Data Availability.}
The author declares that there are no known conflicts of interest associated with this paper. Furthermore, no data are associated with this paper.

\medskip

\noindent{\bf Acknowledgements.\,} The author is very grateful to Professor Yuhua Sun for many helpful discussions on this paper.

\providecommand{\bysame}{\leavevmode\hbox to3em{\hrulefill}\thinspace}


\begin{thebibliography}{10}

\bibitem{ArmstrongSirakov2011CPDE}
S.~N. Armstrong and B.~Sirakov, \emph{Nonexistence of positive supersolutions
  of elliptic equations via the maximum principle}, Comm. Partial Differential
  Equations \textbf{36} (2011), no.~11, 2011--2047.

\bibitem{BidautPohozaev2001JAM}
M.-F. Bidaut-V\'eron and S.~Pohozaev, \emph{Nonexistence results and estimates
  for some nonlinear elliptic problems}, J. Anal. Math. \textbf{84} (2001),
  1--49.

\bibitem{BourgainBrezisMironescu2001}
J.~Bourgain, H.~Brezis, and P.~Mironescu, \emph{Another look at {S}obolev
  spaces}, Optimal control and partial differential equations, IOS, Amsterdam,
  2001, pp.~439--455.

\bibitem{BourgainBrezisMironescu2002JAM}
\bysame, \emph{Limiting embedding theorems for {$W^{s,p}$} when {$s\uparrow1$}
  and applications}, J. Anal. Math. \textbf{87} (2002), 77--101.

\bibitem{CaffarelliSilvestre2007CPDE}
L.~Caffarelli and L.~Silvestre, \emph{An extension problem related to the
  fractional {L}aplacian}, Comm. Partial Differential Equations \textbf{32}
  (2007), no.~7-9, 1245--1260.

\bibitem{CaristiDAmbrosioMitidieri2008}
G.~Caristi, L.~D'Ambrosio, and E.~Mitidieri, \emph{Liouville theorems for some
  nonlinear inequalities}, Tr. Mat. Inst. Steklova \textbf{260} (2008),
  97--118.

\bibitem{CaristiMitidieriPokhozhaev2009}
G.~Caristi, \`E. Mitidieri, and S.~I. Pokhozhaev, \emph{Some {L}iouville
  theorems for quasilinear elliptic inequalities}, Dokl. Akad. Nauk
  \textbf{424} (2009), no.~6, 741--747.

\bibitem{ChenFangYang2015AdvMath}
W.~Chen, Y.~Fang, and R.~Yang, \emph{Liouville theorems involving the
  fractional {L}aplacian on a half space}, Adv. Math. \textbf{274} (2015),
  167--198.

\bibitem{ChenLi2018AdvMath}
W.~Chen and C.~Li, \emph{Maximum principles for the fractional
  {$p$}-{L}aplacian and symmetry of solutions}, Adv. Math. \textbf{335} (2018),
  735--758.

\bibitem{DiCastroKuusiPalatucci2014JFA}
A.~Di~Castro, T.~Kuusi, and G.~Palatucci, \emph{Nonlocal {H}arnack
  inequalities}, J. Funct. Anal. \textbf{267} (2014), no.~6, 1807--1836.

\bibitem{DiCastroKuusiPalatucci2016AIHP}
\bysame, \emph{Local behavior of fractional {$p$}-minimizers}, Ann. Inst. H.
  Poincar\'e{} C Anal. Non Lin\'eaire \textbf{33} (2016), no.~5, 1279--1299.

\bibitem{DiNezzaPalatucciValdinoci2012}
E.~Di~Nezza, G.~Palatucci, and E.~Valdinoci, \emph{Hitchhiker's guide to the
  fractional {S}obolev spaces}, Bull. Sci. Math. \textbf{136} (2012), no.~5,
  521--573.

\bibitem{Filippucci2009NA}
R.~Filippucci, \emph{Nonexistence of positive weak solutions of elliptic
  inequalities}, Nonlinear Anal. \textbf{70} (2009), no.~8, 2903--2916.

\bibitem{FilippucciSunZheng2024JAM}
R.~Filippucci, Y.~Sun, and Y.~Zheng, \emph{A priori estimates and {L}iouville
  type results for quasilinear elliptic equations involving gradient terms}, J.
  Anal. Math. \textbf{153} (2024), no.~1, 367--400.

\bibitem{Gidas1980}
B.~Gidas, \emph{Symmetry properties and isolated singularities of positive
  solutions of nonlinear elliptic equations}, Nonlinear partial differential
  equations in engineering and applied science ({P}roc. {C}onf., {U}niv.
  {R}hode {I}sland, {K}ingston, {R}.{I}., 1979), Lect. Notes Pure Appl. Math.,
  vol.~54, Dekker, New York, 1980, pp.~255--273.

\bibitem{GidasSpruck1981CPAM}
B.~Gidas and J.~Spruck, \emph{Global and local behavior of positive solutions
  of nonlinear elliptic equations}, Comm. Pure Appl. Math. \textbf{34} (1981),
  no.~4, 525--598.

\bibitem{Grafakos-CFAbook}
L.~Grafakos, \emph{Classical {F}ourier {A}nalysis}, third ed., Graduate Texts
  in Mathematics, vol. 249, Springer, New York, 2014.

\bibitem{Grigoryan1985MS}
A.~Grigor'yan, \emph{The existence of positive fundamental solutions of the
  {L}aplace equation on {R}iemannian manifolds}, Mat. Sb. (N.S.)
  \textbf{128(170)} (1985), no.~3, 354--363, 446.

\bibitem{GrigoryanSun2014CPAM}
A.~Grigor'yan and Y.~Sun, \emph{On nonnegative solutions of the inequality
  {$\Delta u+u^\sigma\leq0$} on {R}iemannian manifolds}, Comm. Pure Appl. Math.
  \textbf{67} (2014), no.~8, 1336--1352.

\bibitem{KorvenpaaKuusiLindgren2019JMPA}
J.~Korvenp\"a\"a, T.~Kuusi, and E.~Lindgren, \emph{Equivalence of solutions to
  fractional {$p$}-{L}aplace type equations}, J. Math. Pures Appl. (9)
  \textbf{132} (2019), 1--26.

\bibitem{KorvenpaaKuusiPalatucci2016CVPDE}
J.~Korvenp\"a\"a, T.~Kuusi, and G.~Palatucci, \emph{The obstacle problem for
  nonlinear integro-differential operators}, Calc. Var. Partial Differential
  Equations \textbf{55} (2016), no.~3, Art. 63, 29pp.

\bibitem{KorvenpaaKuusiPalatucci2017MathAnn}
\bysame, \emph{Fractional superharmonic functions and the {P}erron method for
  nonlinear integro-differential equations}, Math. Ann. \textbf{369} (2017),
  no.~3-4, 1443--1489.

\bibitem{KuusiMingioneSire2015CMP}
T.~Kuusi, G.~Mingione, and Y.~Sire, \emph{Nonlocal equations with measure
  data}, Comm. Math. Phys. \textbf{337} (2015), no.~3, 1317--1368.

\bibitem{KuusiMingioneSire2015APDE}
\bysame, \emph{Nonlocal self-improving properties}, Anal. PDE \textbf{8}
  (2015), no.~1, 57--114.

\bibitem{LindgrenLindqvist2014CVPDE}
E.~Lindgren and P.~Lindqvist, \emph{Fractional eigenvalues}, Calc. Var. Partial
  Differential Equations \textbf{49} (2014), no.~1-2, 795--826.

\bibitem{LiuSunXiao2024MathAnn}
L.~Liu, Y.~Sun, and J.~Xiao, \emph{Quasilinear {L}aplace equations and
  inequalities with fractional orders}, Math. Ann. \textbf{388} (2024), no.~1,
  1--60.

\bibitem{LiuXiao2021ACHA}
L.~Liu and J.~Xiao, \emph{Fractional {H}ardy-{S}obolev {$L^1$}-embedding per
  capacity-duality}, Appl. Comput. Harmon. Anal. \textbf{51} (2021), 17--55.

\bibitem{MazyaShaposhnikova2002JFA}
V.~Maz'ya and T.~Shaposhnikova, \emph{On the {B}ourgain, {B}rezis, and
  {M}ironescu theorem concerning limiting embeddings of fractional {S}obolev
  spaces}, J. Funct. Anal. \textbf{195} (2002), no.~2, 230--238.

\bibitem{MitidieriPokhozhaev1998}
\`E. Mitidieri and S.~I. Pokhozhaev, \emph{Absence of global positive solutions
  of quasilinear elliptic inequalities}, Dokl. Akad. Nauk \textbf{359} (1998),
  no.~4, 456--460.

\bibitem{MitidieriPokhozhaev1999}
\bysame, \emph{Absence of positive solutions for quasilinear elliptic problems
  in {${\bf R}^N$}}, Tr. Mat. Inst. Steklova \textbf{227} (1999), 192--222.

\bibitem{MitidieriPokhozhaev2001}
\bysame, \emph{A priori estimates and the absence of solutions of nonlinear
  partial differential equations and inequalities}, Tr. Mat. Inst. Steklova
  \textbf{234} (2001), 1--384.

\bibitem{NiSerrin1986CPAM}
W.-M. Ni and J.~Serrin, \emph{Nonexistence theorems for singular solutions of
  quasilinear partial differential equations}, Comm. Pure Appl. Math.
  \textbf{39} (1986), no.~3, 379--399.

\bibitem{SchikorraShiehSpector2018CCM}
A.~Schikorra, T.-T. Shieh, and D.~E. Spector, \emph{Regularity for a fractional
  {$p$}-{L}aplace equation}, Commun. Contemp. Math. \textbf{20} (2018), no.~1,
  Page No. 1750003, 6pp.

\bibitem{SerrinZou2002Acta}
J.~Serrin and H.~H. Zou, \emph{Cauchy-{L}iouville and universal boundedness
  theorems for quasilinear elliptic equations and inequalities}, Acta Math.
  \textbf{189} (2002), no.~1, 79--142.

\bibitem{ShiXiao2017CVPDE}
S.~Shi and J.~Xiao, \emph{Fractional capacities relative to bounded open
  {L}ipschitz sets complemented}, Calc. Var. Partial Differential Equations
  \textbf{56} (2017), no.~1, Paper No. 3, 22pp.

\bibitem{ShiehSpector2015ACV}
T.-T. Shieh and D.~E. Spector, \emph{On a new class of fractional partial
  differential equations}, Adv. Calc. Var. \textbf{8} (2015), no.~4, 321--336.

\bibitem{ShiehSpector2018ACV}
\bysame, \emph{On a new class of fractional partial differential equations
  {II}}, Adv. Calc. Var. \textbf{11} (2018), no.~3, 289--307.

\bibitem{Silvestre2007CPAM}
L.~Silvestre, \emph{Regularity of the obstacle problem for a fractional power
  of the {L}aplace operator}, Comm. Pure Appl. Math. \textbf{60} (2007), no.~1,
  67--112.

\bibitem{Sun2015PAMS}
Y.~Sun, \emph{On nonexistence of positive solutions of quasi-linear inequality
  on {R}iemannian manifolds}, Proc. Amer. Math. Soc. \textbf{143} (2015),
  no.~7, 2969--2984.

\bibitem{Sun2016PJM}
\bysame, \emph{Uniqueness result on nonnegative solutions of a large class of
  differential inequalities on {R}iemannian manifolds}, Pacific J. Math.
  \textbf{280} (2016), no.~1, 241--254.

\bibitem{WangXiao2016AIHP}
Y.~Wang and J.~Xiao, \emph{A constructive approach to positive solutions of
  {$\Delta_pu+f(u,\nabla u)\leq 0$} on {R}iemannian manifolds}, Ann. Inst. H.
  Poincar\'e{} C Anal. Non Lin\'eaire \textbf{33} (2016), no.~6, 1497--1507.

\bibitem{WangXiao2016CCM}
\bysame, \emph{A uniqueness principle for {$u^p\leq(-\Delta)^{\frac\alpha 2}u$}
  in the {E}uclidean space}, Commun. Contemp. Math. \textbf{18} (2016), no.~6,
  Page No. 1650019, 17pp.


\bibitem{Zou1995DIE}
H.~H. Zou, \emph{Slow decay and the {H}arnack inequality for positive solutions
  of {$\Delta u+u^p=0$} in {${\bf R}^n$}}, Differential Integral Equations
  \textbf{8} (1995), no.~6, 1355--1368.

\end{thebibliography}
\end{document}